\documentclass[reqno,10pt]{amsart}
\usepackage{amsmath}
\usepackage{amssymb}
\usepackage{amsfonts}
\usepackage{a4}
\usepackage{graphicx}
\usepackage[abs]{overpic}
\usepackage{caption}
\usepackage{wrapfig}
\usepackage{float}
\usepackage{graphicx}
\usepackage{mathrsfs}
\usepackage{accents}
\usepackage{amsthm}

\usepackage{tikz-cd}
\usepackage[colorinlistoftodos]{todonotes}

\newtheorem{theorem}{Theorem}[section]
\newtheorem{lemma}[theorem]{Lemma}

\newtheorem{definition}[theorem]{Definition}

\newtheorem{rem}[theorem]{Remark}

\newcommand{\Proof}{\par\noindent{\em Proof. }}
\newcommand{\eop}{\nopagebreak\hspace*{\fill}$\Box$\smallskip}

\newcommand{\N}{\Bbb N}

\newcommand{\R}{\Bbb R}
\newcommand{\C}{\Bbb C}

\newcommand{\stress}{{T^E}}
\newcommand{\vstress}{{S}}

\def\Id{\mathbf{Id}}
\def\id{\mathbf{id}}

\def\eps{\varepsilon}

\def\dist{\operatorname{dist}}

\def\XXint#1#2#3{{\setbox0=\hbox{$#1{#2#3}{\int}$}
     \vcenter{\hbox{$#2#3$}}\kern-.5\wd0}}

\usepackage{color}

\newcommand{\BBB}{\color{black}} 
\newcommand{\EEE}{\color{black}} 

\newcommand{\MK}{\color{black}}

\title[On the passage from nonlinear to linearized viscoelasticity]{On the passage from nonlinear to linearized viscoelasticity}

\author{Manuel Friedrich}
\author{Martin Kru\v{z}\'ik}


\MK
\address[Manuel Friedrich]{Institute for Computational and Applied Mathematics
  University of  M\"{u}nster, Einsteinstr.~62, D-48149 M\"{u}nster, Germany}
\email{manuel.friedrich@uni-muenster.de}
\EEE 

\address[Martin Kru\v{z}\'ik]{Czech Academy of Sciences, Institute of Information Theory and Automation
Pod vod\'arenskou v\v{e}\v{z}\'i 4, CZ-182 08 Praha 8, Czechia
(corresponding address) \& Faculty of Civil Engineering, Czech Technical University, Th\'akurova 7, CZ-166 29 Praha 6, Czechia}
\email{kruzik@utia.cas.cz}

\begin{document}

\maketitle

\begin{center}
\today
\end{center}
\bigskip

\begin{abstract}
We formulate a quasistatic nonlinear  model  for nonsimple viscoelastic materials at a finite-strain setting in the Kelvin's-Voigt's rheology where the   viscosity stress tensor   complies with the principle of time-continuous frame-indifference.  We identify weak solutions in the nonlinear framework as  limits of time-incremental problems for vanishing time increment. Moreover,   we show that  linearization around  the  identity leads to the standard system  for  linearized viscoelasticity and that solutions of the nonlinear system converge in a suitable sense to solutions of the linear one.  The  same property holds for time-discrete approximations and  we provide a corresponding commutativity result.  Our main tools are the theory of gradient flows  in metric spaces and $\Gamma$-convergence.
\end{abstract}
\bigskip

%\tableofcontents

\section{Introduction}
%\todo[inline]{ To discuss: boundary conditions and second gradient and E-L equations, polyconvexity not needed, the paper supports nonsimple materials}

Neglecting inertia,  a    nonlinear viscoelastic material in Kelvin's-Voigt's rheology obeys the following system of equations 
\begin{align}\label{eq:viscoel}-{\rm div}\Big(\partial_FW(\nabla y)   + \partial_{\dot F}R(\nabla y,\partial_t \nabla y)  \Big) =  f\text{ in $  [0,T] \times \EEE \Omega$.} \end{align}
Here, $[0,T]$ is a process time interval with $T>0$,   $\Omega\subset\R^d$  ($d=2$ or $d=3$)  is a smooth bounded domain representing the reference configuration,  and \EEE  $y:[0,T]\times \Omega\to\R^d$ is a deformation mapping  with corresponding \EEE  deformation gradient $\nabla y$.  Further, \EEE $W:\R^{d\times d}\to  [0,\infty]\EEE$ is a stored energy density, which represents a potential of the first Piola-Kirchhoff stress tensor $\stress$, i.e.,  $\stress:=\partial_F W:=\partial W/\partial F$ and  $F\in\R^{d\times d}$ is the placeholder of $\nabla y$. Finally, $R:  \R^{d \times d} \times \R^{d \times d} \to [0,\infty) \EEE $ denotes a (pseudo)potential of dissipative forces,  where $\dot F \in \R^{d \times d}$ is the placeholder of $\partial_t \nabla y$, \EEE  and $f:[0,T]\times \Omega\to\R^d$ is a volume density of external forces acting on $\Omega$. 
In the present contribution, we consider a version of \eqref{eq:viscoel} for nonsimple materials where the elastic stored energy density depends also on the second gradient  of $y$.  In this case, we get 

\begin{align}\label{eq:viscoel-nonsimple}
-{\rm div}\Big( \partial_F W(\nabla y) + \eps\mathcal{L}_{P}(\nabla^2 y)  + \partial_{\dot{F}}R(\nabla y,\partial_t \nabla y)  \Big) =  f\text{ in $\BBB [0,T] \times \EEE \Omega$,} \end{align}
where $\eps>0$ is small and  $\mathcal{L}_{P}$ is a  first \EEE order differential operator which is associated to an additional term $\int_\Omega P(\nabla^2 y)$   in the stored elastic energy, e.g., for $P(G):=  \frac{1}{2} \EEE |G|^2$ with $G\in\R^{d\times d\times d}$, we get $-{\rm div}\mathcal{L}_{P}(\nabla^2 y)= \Delta^2 y$.   We refer to  \eqref{LP-def} for more details. Thus, we resort to the so-called nonsimple materials, the stored energy density (and the  first Piola-Kirchhoff stress tensor, too) of which depends also on the second gradient of the deformation. This idea was first introduced by Toupin  \cite{Toupin:62,Toupin:64} and proved to be useful in mathematical elasticity, \MK see e.g.~\cite{BCO,Batra, chen, MielkeRoubicek:16,MielkeRoubicek,Podio}  \EEE because it brings additional compactness to the problem.      The first Piola-Kirchhoff stress tensor, $\stress$,  then reads for all $i,j\in\{1,\ldots, d\}$
\begin{align*}
\stress_{ij}(F,G):= \MK \partial_{F_{ij}} W(F) +\eps \big(\mathcal{L}_{P}(G)\big)_{ij}\EEE = \partial_{F_{ij}} W(F) -\eps\sum_{k=1}^d \partial_k \big(\partial_{G_{ijk}}P(G)\big),
\end{align*}
where $G\in\R^{d\times d\times d}$  is the placeholder for the second gradient of $y$.  The term 
$\eps\partial_{G}P(G)$ is usually called  hyperstress.

We standardly assume that $W$ as well as $P$ are frame-indifferent functions, i.e., that $W(F)=W(QF)$  and $P(G)=P(QG)$ for every proper rotation $Q\in{\rm SO}(d)$, every $F\in\R^{d\times d}$, and every $G\in\R^{d\times d\times d}$.  This implies that $W$ depends on the right Cauchy-Green strain tensor  $C:=F^\top F$, see e.g.~\cite{Ciarlet}. \BBB  We wish to emphasize that, in the case of nonsimple materials, \EEE no convexity properties of $W$ are needed, in particular, we do not have to assume that $W$ is polyconvex \cite{Ball:77,Ciarlet}.  Moreover, it is shown in \cite{HealeyKroemer:09} that if $W$ satisfies suitable and physically relevant  growth conditions (as $W(F)\to\infty$ if ${\rm det}\, F\to 0$),   then every minimizer of the elastic energy is a weak solution to the corresponding Euler-Lagrange equations.  

The second term on the left-hand side of \eqref{eq:viscoel} is the \MK viscous \EEE stress tensor 
$\vstress(F,\dot F):= \partial_{\dot F} R(F,\dot F)$ which has its origin in viscous dissipative mechanisms of the material.  Notice that its potential $R$ plays an  analogous role as $W$ in the case of purely elastic, i.e., non-dissipative processes. Naturally, we require that $R(F,\dot F)\ge R(F,0)=0$.  The viscous stress tensor must comply with the time-continuous frame-indifference principle  meaning that for all $F$
\begin{align*}
\vstress(F,\dot F)=F\tilde\vstress(C,\dot C)
\end{align*}
where $\tilde\vstress$ is a symmetric matrix-valued function.   This condition constraints 
$R$ so that \cite{Antmann, Antmann:04,MOS} \MK (see also \cite{Demoulini}) \EEE
\begin{align}\label{eq:frame indifference-R}
R(F,\dot F)=\tilde R(C,\dot C)\ ,
\end{align}
for some nonnegative function $\tilde R$.  In other words, \EEE $R$ must depend on the right Cauchy-Green strain tensor $C$ and its time derivative $\dot C$. 

In this work, we are interested in the case of small strains, i.e., when  $\nabla u:=\nabla y - \Id$ is of order $\delta$ for some small $\delta >0$. Here, $u:=y-\id$ is the displacement corresponding to $y$ with  $\id$ and $\Id$ standing for the identity map and identity matrix, respectively. Such a property is certainly meaningful if one considers initial values $y_0$ with $\Vert \nabla y_0 - \Id \Vert_{L^2(\Omega)} \le \delta$. Therefore, it is convenient to define the rescaled displacement $u = \delta^{-1}(y - \id)$. Introducing a proper scaling in the above equation we get
\begin{align}\label{eq:viscoel-nonsimple-scaled}
-{\rm div}\Big( \delta^{-1}\partial_F W(\Id + \delta \nabla u) + \tilde{\eps}\mathcal{L}_{P}(\delta\nabla^2 u)  + \delta^{-1}\partial_{\dot{F}}R(\Id + \delta \nabla u, \delta \partial_t \nabla u)  \Big) = f\end{align}
for $\tilde{\eps}=\tilde{\eps}(\delta)$ appropriate. \MK Note that to obtain \eqref{eq:viscoel-nonsimple-scaled} from \eqref{eq:viscoel} we write the latter equation for $f:=\delta f$ and then divide the whole equation by $\delta$. \EEE
 Formally, we can pass to the limit and obtain the equation (for $\tilde{\eps} \to 0$ as $\delta \to 0$)
\begin{align}\label{eq:viscoel-small} -{\rm div}\Big( \C_W e(u) + \C_D  e(\partial_t u)   \Big) = f,
 \end{align}
 where $\C_W:=\partial^2_{F^2}W(\Id)$ is the tensor of elastic constants, $\C_D:= \partial^2_{\dot F^2}R(\Id,0)$ is the tensor of viscosity coefficients,  and $e(u):=(\nabla u+(\nabla u)^\top)/2$ denotes the linear strain tensor. 

 The goal of this contribution is twofold: we first show existence of solutions \MK to the nonlinear system of equations \eqref{eq:viscoel-nonsimple-scaled}. \EEE  Afterwards, \EEE we make the limit passage rigorous, i.e., we show that solutions to the nonlinear equations converge to  the unique \EEE solution of the linear systems as $\delta\to 0$.  Interestingly, 
although the nonlinear viscoelastisity systems is written for a nonsimple material, in the limit we obtain the  standard \EEE linear equations without spatial gradients of $e(u)$.

Our general strategy is to treat the system of quasistatic viscoelasticity in the abstract setting of metric gradient flows \cite{AGS} which was, to our best knowledge, formulated  for the first time in  \cite{MOS} for simple materials (i.e.~only the first gradient of $y$ is considered).
 However,  in their setting, \EEE a passage from time-discrete problems to a continuous one is  only 
possible in a specific  one-dimensional case. See  also \cite{BallSenguel:15} for a related approach in materials undergoing phase transition. This, in our opinion,  also supports models of nonsimple  materials as their linearization \BBB leads to \EEE the usual small-strain viscoelasticity model which seems unreachable (or at least rather difficult) in the case of simple materials.

\BBB
An abstract framework for the study of metric gradient flows   along a sequence of energies and metric spaces  has been developed in \cite{S1,S2}. In practice, for each specific problem the challenge lies in proving that the additional conditions needed to ensure convergence of gradient flows are satisfied  (we refer to \cite{S2} for some examples in that direction). Our aim is to show that the passage of nonlinear to linearized viscoelasticity can be formulated in this setting.  Let us also mention that a rigorous analysis of  the static, purely elastic case without viscosity goes back to \cite{DalMasoNegriPercivale:02}. 
\EEE

 \MK
Heuristically, the idea of gradient flows in metric spaces stems from the observation that,  having a Hilbert space  (equipped with the dot product $\langle\cdot,\cdot\rangle$),  the inequality 
$$
|u'|^2 +2\langle u',\nabla\phi(u)\rangle+|\nabla\phi(u)|^2\ge 0  
$$
 becomes equality if and only if
$$ u'=-\nabla \phi(u),
$$
i.e., if $u$ solves the gradient flow equation.
This approach can be extended to metric spaces provided we are able to find analogies to $|u'|$ and $|\nabla\phi|$ in metric spaces. These are called the metric derivative and the upper gradient (or slope), respectively.  Precise definitions can be found in Section~\ref{sec: defs} below.
\EEE

The plan of the paper is as follows.  In Section~\ref{sec:Model}, we introduce the nonlinear and linear  systems   of  viscoelasticity \EEE in more detail and state our main results.  In particular, Theorem~\ref{maintheorem1} and Theorem~\ref{maintheorem2} show the existence of  solutions to the nonlinear and linear problems, respectively. These solutions can be identified with  so-called  \emph{curves of maximal slope} \BBB introduced in \cite{DGMT}. \EEE   Proofs of  existence rely on semidiscretization in time, and on the theory of \emph{generalized minimizing movements} and gradient flows in metric spaces \cite{AGS},  where the underlying metric is given by a \emph{dissipation distance} suitably related to the potential $R$ (see \eqref{intro:R}). \EEE  Finally, Theorem~\ref{maintheorem3} shows the relationship between the two systems. Besides convergence of solutions of \eqref{eq:viscoel-nonsimple} to solutions of \eqref{eq:viscoel-small}, we also get analogous convergences for semidiscretized problems.  Moreover, convergences for vanishing time step and $\delta\to 0$ commute, see Figure~\ref{diagram}.  (For a related commutativity result in an abstract setting we refer to \cite{BCGS}.)

 Section~\ref{sec3} is devoted  to definitions of generalized minimizing movements (GMM) and curves of maximal slope.  \BBB Here we also collect the necessary existence results proved in \cite{AGS}. Moreover, we present a statement similar to \cite{Ortner,S2} about  sequences of curves of maximal slope and their limits as well as a corresponding result for minimizing movements.    \EEE

  Further, Section~\ref{sec:energy-dissipation} shows interesting properties of dissipation distances related to our viscous dissipation.  It turns out that by frame indifference \eqref{eq:frame indifference-R} the dissipation distances are genuinely non-convex. However, due to the presence of the higher order gradient we are able to obtain sufficiently good convexity properties in order to apply the abstract theory  \BBB \cite{AGS, S2}. \EEE Finally, proofs of our results can be found in Section~\ref{sec results}.  In particular, we relate curves of maximal slope for the nonlinear system with limiting curves of maximal slope as $\delta\to 0$ and identify these configurations as weak solutions of \eqref{eq:viscoel-nonsimple} and \eqref{eq:viscoel-small}. \EEE

 In what follows, we use standard notation for Lebesgue spaces,  $L^p(\Omega)$, which are measurable maps on $\Omega\subset\R^d$ integrable with the $p$-th power (if $1\le p<+\infty$) or essentially bounded (if $p=+\infty$).      Sobolev spaces, i.e., $W^{k,p}(\Omega)$ denote the linear spaces of  maps  which, together with their derivatives up to the order $k\in\N$, belong to $L^p(\Omega)$. \BBB Further,  $W^{k,p}_0(\Omega)$ contains maps from $W^{k,p}(\Omega)$ having zero boundary conditions (in the sense of traces). \EEE In order to emphasize its Hilbert structure, we write $H^1(\Omega):=W^{1,2}(\Omega)$.  We also  work with  the dual space to $H^1_0(\Omega)$ denoted by $H^{-1}(\Omega)$. We refer to \cite{AdamsFournier:05} for more details on Sobolev spaces and their duals. 

If $A\in\R^{d\times d\times d\times d}$ and $e\in\R^{d\times d}$ then $Ae\in\R^{d\times d}$ such that for $i,j\in\{1,\ldots, d\}$ we define  $(Ae)_{ij}:=A_{ijkl}e_{kl}$ where we use Einstein's summation convention. An analogous  convention is used in similar  occasions, in the sequel.
Finally, at many spots, we  follow closely notation introduced in \cite{AGS} to ease readability of our work because the theory developed there is one of the main tools of our analysis.

\section{The model and main results}\label{sec:Model}

\subsection{The nonlinear setting}

We adopt the usual setting of nonlinear elasticity: consider $\Omega \subset \R^d$ open, bounded with Lipschitz boundary. Fix $\delta>0$ (small), $p>d$ and  $0< \alpha<1$. The parameter $\tilde\eps(\delta)$ introduced in \eqref{eq:viscoel-nonsimple-scaled} is defined as $\tilde\eps(\delta):=\delta^{1-p\alpha}$.  

\textbf{Stored elastic energy and body forces:} We introduce the nonlinear elastic energy $\phi_\delta: W^{2,p}(\Omega;\R^d) \to [0,\infty]$ by 
\begin{align}\label{nonlinear energy}
\phi_\delta(y) = \frac{1}{\delta^2}\int_\Omega W(\nabla y(x))\, dx + \frac{1}{\delta^{p\alpha}}\int_\Omega P(\nabla^2 y(x)) \, dx - \frac{1}{\delta}\int_\Omega f(x)\cdot y(x) \, dx
\end{align}
for a \emph{deformation} $y: W^{2,p}(\Omega;\R^d) \to \R^d$. Here, $W: \R^{d \times d} \to [0,\infty]$ is a single well, frame indifferent stored energy functional with the usual assumptions in nonlinear elasticity. Altogether, we suppose  that there exists $c>0$ such that
\begin{align}\label{assumptions-W}
\begin{split}
(i)& \ \ W \text{ continuous and $C^3$ in a neighborhood of $SO(d)$},\\
(ii)& \ \ \text{Frame indifference: } W(QF) = W(F) \text{ for all } F \in \R^{d \times d}, Q \in SO(d),\\
(iii)& \ \ W(F) \ge c\dist^2(F,SO(d)), \  W(F) = 0 \text{ iff } F \in SO(d),
\end{split}
\end{align}
where $SO(d) = \lbrace Q\in \R^{d \times d}: Q^\top Q = \Id, \, \det Q=1 \rbrace$. Moreover, $P: \R^{d\times d \times d} \to [0,\infty]$ denotes a higher order perturbation satisfying 
\begin{align}\label{assumptions-P}
\begin{split}
(i)& \ \ \text{frame indifference: } P(QG) = P(G) \text{ for all } G \in \R^{d \times d \times d}, Q \in SO(d),\\
(ii)& \ \ \text{$P$ is convex and $C^1$},\\
(iii)& \ \ \text{growth condition: For all $G \in \R^{d \times d \times d}$ we have } \\&   \ \ \ \ \ \    c_1 |G|^p \le P(G) \le c_2 |G|^p, \ \ \ \ \ \   \BBB |\partial_G P(G)| \EEE \le c_2 |G|^{p-1} \EEE
\end{split}
\end{align}
for $0<c_1<c_2$. Finally, $f \in L^\infty(\Omega;\R^d)$ denotes a volume force.  From now on we always drop  the target space $\R^d$ for notational convenience when no confusion arises. \EEE  We remark that by minor  adaptions of our arguments we can also treat potentials with additional dependence on the material point $x \in \Omega$. We scale the energy appropriately with a (small) positive parameter $\delta$ as we will eventually be interested in the behavior in the small strain limit $\delta \to 0$.  

\textbf{Dissipation potential and viscous stress:} Consider a time dependent deformation $y: [0,T] \times \Omega \to \R^d$. Viscosity is not only related to the strain $\nabla y(t,x)$, but also to the strain rate $\partial_t \nabla  y(t,x)$ and can be expressed in terms of a  dissipation potential $R(\nabla y, \partial_t \nabla y)$, where $R: \R^{d \times d} \times \R^{d \times d} \to [0,\infty)$. An admissible potential has to satisfy frame indifference in the sense (see \cite{Antmann, MOS})
\begin{align}\label{R: frame indiff}
R(F,\dot{F}) = R(QF,Q(\dot{F} + AF))  \ \ \  \forall  Q \in SO(d), A \in {\rm Skew}(d)
\end{align}
for all $F \in GL_+(d)$ and $\dot{F} \in \R^{d \times d}$, where $GL_+(d) = \lbrace F \in \R^{d \times d}: \det F>0 \rbrace$ and ${\rm Skew}(d) = \lbrace A  \in \R^{d \times d}: A=-A^\top \rbrace$. 

Following the discussion in \cite[Section 2.2]{MOS}, from the point of modeling  it is much more  convenient to postulate the existence of a (smooth) global distance $D: GL_+(d) \times GL_+(d) \to [0,\infty)$ satisfying $D(F,F) = 0$ for all $F \in GL_+(d)$, from which an associated dissipation potential $R$ can be calculated by
\begin{align}\label{intro:R}
R(F,\dot{F}) := \lim_{\eps \to 0} \frac{1}{2\eps^2} D^2(F+\eps\dot{F},F) = \frac{1}{4} \partial^2_{F_1^2} D^2(F,F) [\dot{F},\dot{F}]
\end{align}
for $F \in GL_+(d)$, $\dot{F} \in \R^{d \times d}$,  where $\partial^2_{F_1^2} D^2(F_1,F_2)$ denotes the Hessian of  $D^2$ in direction of $F_1$ at $(F_1,F_2)$, being a fourth order tensor.   We have the following assumptions on $D$ for some $c>0$.
\begin{align}\label{eq: assumptions-D}
(i) & \ \ D(F_1,F_2)> 0 \text{ if } F_1^\top F_1 \neq F_2^\top F_2,\notag \\
(ii) & \ \ D(F_1,F_2) = D(F_2,F_1),\\
(iii) & \ \ D(F_1,F_3) \le D(F_1,F_2) + D(F_2,F_3),\notag \\
(iv) & \ \ \text{$D(\cdot,\cdot)$ is $C^3$ in a neigborhood of $SO(d) \times SO(d)$},\notag 
\\
(v)& \ \ \text{Separate frame indifference: } D(Q_1F_1,Q_2F_2) = D(F_1,F_2)\notag \\
& \ \  \ \ \ \ \ \ \ \ \    \ \  \ \ \ \ \ \ \ \ \    \ \  \ \ \ \ \ \ \ \ \    \ \  \ \ \ \ \ \ \ \ \   \forall Q_1,Q_2 \in SO(d),  \ \forall F_1,F_2 \in GL_+(d),\notag\\ 
(vi) & \ \ \text{$D(F,\Id) \ge c\dist(F,SO(d))$  $\forall F \in \R^{d \times d}$ in a neighborhood of $SO(d)$}.\notag
\end{align}
Note that conditions (i),(iii) state that $D$ is a true distance when restricted to symmetric matrices. We can not expect more due to the separate frame indifference (v). We also note that (v) implies \eqref{R: frame indiff} as shown in \cite[Lemma 2.1]{MOS}. Note that in our model we do not require any conditions of polyconvexity neither for $W$ nor for $D$ \cite{Ball:77}.  For examples of admissible dissipation distances we refer the reader to \cite[Section 2.3]{MOS}.   

\textbf{Equations of nonlinear viscoelasticity:} We will impose the boundary conditions $y(t,x) = x$ for $(t,x) \in [0,T] \times \partial \Omega$  and for convenience we define  the set $W^{2,p}_\id(\Omega)= \lbrace y = \id + u \in W^{2,p}(\Omega): u \in W^{2,p}_0(\Omega) \rbrace$, where $\id$ denotes the identity function on $\Omega$.
 We remark that our results can be extended to  more general Dirichlet boundary conditions, too, which we do not include here for the sake of maximizing simplicity rather than generality. We now introduce a differential operator associated to the perturbation $P$ (cf. \eqref{assumptions-P}). To this end, we use the notation  $(\nabla y)_{ik} = \partial_k y_i$ and $(\nabla^2 y)_{ijk} = \partial^2_{jk} y_i$ for $i,j,k \in \lbrace 1,\ldots,d\rbrace$ \EEE and define
\begin{align}\label{LP-def}
\big(\mathcal{L}_P(\nabla^2 y)\big)_{ij} =  -\sum\nolimits_{k=1}^d  \partial_k(\partial_GP(\nabla^2 y))_{ijk}, \ \ \ \ \BBB i,j \EEE \in \lbrace 1,\ldots, d\rbrace
\end{align}   
for $y \in W^{2,p}_\id(\Omega)$, where the derivatives have to be understood in the sense of distributions. The equations of nonlinear viscoelasticity then read as (respecting the different scalings of the terms in \eqref{nonlinear energy})
\begin{align}\label{nonlinear equation}
\begin{split}
\begin{cases} -  {\rm div}\Big( \partial_FW(\nabla y) + \delta^{2-p\alpha}\mathcal{L}_{P}(\nabla^2 y)  + \partial_{\dot{F}}R(\nabla y,\partial_t \nabla y)  \Big) = \delta f  & \text{in } [0,\infty) \times \Omega \\
y(0,\cdot) = y_0 & \text{in } \Omega \\
y(t,\cdot) \in W^{2,p}_\id(\Omega) &\text{for } t\in [0,\infty)
\end{cases}
\end{split}
\end{align}
for some $y_0 \in W^{2,p}_\id(\Omega)$, where $\partial_FW(\nabla y(t,x))$ denotes the first  \emph{Piola-Kirchhoff stress tensor} and $\partial_{\dot{F}}R(\nabla y(t,x),\partial_t \nabla y(t,x))$ the \emph{viscous stress} with $R$ as introduced in \eqref{intro:R}. 
 The first goal of the present contribution is  to prove the existence of weak solutions to \eqref{nonlinear equation}.  More precisely, we say that $y \in L^\infty([0,\infty);W^{2,p}_{\id}(\Omega)) \cap W^{1,2}([0,\infty);H^1(\Omega))$ is a \emph{weak solution} of \eqref{nonlinear equation} if $y(0,\cdot) = y_0$ and for a.e. $t \ge 0$
 \begin{align}\label{nonlinear equation2}
 \begin{split}
& \int_\Omega \Big( \partial_FW(\nabla y(t,x)) +  \partial_{\dot{F}}R(\nabla y(t,x),\partial_t \nabla y(t,x))\Big)  : \nabla \varphi(x) \, dx  \\
& \ \ \ \  \ \ \ \ \  + \int_\Omega\delta^{2-p\alpha} \partial_GP(\nabla^2 y(t,x)) :\nabla^2 \varphi(x) \, dx  = \delta \int_\Omega  f(x) \cdot \varphi(x) \, dx  
\end{split}
 \end{align}
 for all $\varphi \in W^{2,p}_0(\Omega)$. In particular, we note that the first term in the second line is well defined for a weak solution by \eqref{assumptions-P}(iii) and H\"older's inequality.

 \EEE

\subsection{The linear problem} 
After rescaling with $\delta^{-1}$ and introducing the rescaled displacement field $u(t,x) = \delta^{-1} (y(t,x)-x)$, the partial differential equation \eqref{nonlinear equation} can be written as 
$$-{\rm div}\Big( \delta^{-1}\partial_FW(\id + \delta \nabla u) + \delta^{1-p\alpha}\mathcal{L}_{P}(\delta\nabla^2 u)  + \delta^{-1}\partial_{\dot{F}}R(\id + \delta \nabla u, \delta \partial_t \nabla u)  \Big) = f$$
with an initial datum $u_0 = \delta^{-1}(y_0 - \id)$. For $\alpha$ small, letting $\delta \to 0$ we obtain, at least formally, the equation 
\begin{align}\label{linear equation}
\begin{split}
\begin{cases} -{\rm div}\Big( \C_W e(u) + \C_D  e(\partial_t u)   \Big) = f   & \text{in } [0,\infty) \times \Omega \\
u(0,\cdot) = u_0 & \text{in } \Omega \\
u(t,\cdot) \in H^1_{0}(\Omega) &\text{for } t\in [0,\infty),
\end{cases}
\end{split}
\end{align}
where $\C_W := \partial^2_{F^2} W(\Id)$ and $\C_D :=  \frac{1}{2}\partial^2_{F_1^2} D^2(\Id,\Id)  $ (cf. \eqref{intro:R}). Note that the frame indifference of the energy and the dissipation (see \eqref{assumptions-W}(ii) and \eqref{eq: assumptions-D}(v), respectively) imply that the contributions only depend on the symmetric part of the strain $e(u) = \frac{1}{2}(  \nabla u  +(\nabla u)^\top)$ and the strain rate  $e(\partial_t u) = \frac{1}{2}( \partial_t \nabla u + \partial_t (\nabla u)^\top)$. Let us also mention that the stress tensor is related to the linearized elastic energy $\bar{\phi}_0 : H_0^1(\Omega) \to [0,\infty)$ given by
\begin{align}\label{linear energy}
\bar{\phi}_0(u) = \int_\Omega \frac{1}{2}\C_W[e(u)(x),  e(u)(x)] \, dx - \int_\Omega f(x) \cdot u(x) \,dx
\end{align}
for $u \in H^1_0(\Omega)$. The  goal of this article is to show that the above reasoning can be made rigorous: we will prove that \eqref{linear equation} admits a unique weak solution and that solutions of \eqref{nonlinear equation} converge to the solution of \eqref{linear equation} in a suitable sense.  Here, similarly as before, we say $u \in W^{1,2}([0,\infty); H^1_0(\Omega))$ is a \emph{weak solution} of \eqref{linear equation}   if $u(0,\cdot) = u_0$ and for a.e. $t \ge 0$ and all $\varphi \in H^{1}_0(\Omega)$ we have 
$$\int_\Omega  ( \C_W e(u) + \C_D  e(\partial_t u) )  : \nabla \varphi   =  \int_\Omega  f\cdot \varphi$$.  \EEE

\subsection{Main results}
Let us introduce the \emph{global dissipation distance} between two deformations for the nonlinear and linear setting by
\begin{align}\label{eq: D,D0}
\begin{split}
 &\mathcal{D}_\delta(y_0,y_1) = \delta^{-1}\Big(\int_\Omega D^2(\nabla y_0, \nabla y_1) \Big)^{1/2}, \\
 & \bar{\mathcal{D}}_0(u_0,u_1) = \Big(\int_\Omega \C_D[\nabla u_0 - \nabla u_1,\nabla u_0 - \nabla u_1]   \Big)^{1/2}
 \end{split}
\end{align}
for $y_0,y_1 \in W^{2,p}_\id(\Omega)$ and $u_0,u_1 \in H^1_0(\Omega)$, respectively. (In many notations we include an overline to indicate that the notion is related to the linear setting.) We also define the sublevel sets $\mathscr{S}_\delta^M := \lbrace y\in W^{2,p}_\id(\Omega): \phi_\delta(y) \le M\rbrace$. (For convenience we do not include $\Omega$ in the notation.)  Our general strategy will be to show that the spaces $(\mathscr{S}_\delta^M, \mathcal{D}_\delta)$ and $(H^1_0(\Omega), \bar{\mathcal{D}}_0)$ are complete metric spaces and to follow the approach in \cite{AGS} (see Theorem \ref{th: metric space} and Theorem \ref{th: metric space-lin} below).

 In particular, to show existence of solutions to the problems \eqref{nonlinear equation} and \eqref{linear equation},  we will apply an approximation scheme solving suitable time-incremental minimization problems and show that \BBB time-continuous \EEE limits are  curves of maximal slope for the elastic energies $\phi_\delta, \bar{\phi}_0$, respectively. Finally, using the property that in Hilbert spaces curves of maximal slope can be related to gradient flows, we find solutions to \eqref{nonlinear equation}, \eqref{linear equation}.

 Moreover, to study the relation between the nonlinear and linear problem we will apply some results about the limit of sequences of curves of maximal slope proved in Section \ref{sec: auxi-proofs}.
 
 For the main definitions and notation for discrete solutions, (generalized) minimizing movements   (abbreviated by MM and GMM, see Definition~\ref{main def1})  and curves of maximal slope we refer to Section \ref{sec: defs}. In particular, we define $\Phi_\delta$ and $\bar{\Phi}_0$, respectively, as in \eqref{incremental} replacing $\phi, \mathcal{D}$ by $\phi_\delta,\mathcal{D}_\delta$ and $\bar{\phi}_0,\bar{\mathcal{D}}_0$, respectively. Moreover, we write $|\partial \phi_\delta|_{\mathcal{D}_\delta}$, $|\partial \bar{\phi}_0|_{\bar{\mathcal{D}}_0}$ for the (local) slopes and $|y'|_{\mathcal{D}_\delta}$, $|u'|_{\bar{\mathcal{D}}_0}$ for the metric derivatives, respectively (see Definition \ref{main def2}). Finally, discrete solutions for time step $\tau > 0$ will be denoted by \MK  $\tilde{Y}^\delta_\tau$ and $\tilde{U}^0_\tau$, respectively. \EEE

Our first main result addresses the existence of solutions to the nonlinear problem.

\begin{theorem}[Solutions to the nonlinear problem]\label{maintheorem1}
Let $M>0$ and $\mathscr{S}_\delta^M = \lbrace y \in W^{2,p}_\id(\Omega): \phi_\delta(y) \le M\rbrace$. Then for $\delta>0$ sufficiently small only depending on $M$ the following holds:

(i)   (Existence of GMM) \EEE $GMM(\Phi_\delta;y_0) \neq \emptyset$ for all $y_0 \in \mathscr{S}^M_\delta$.

(ii)  (Curves of maximal slope) \EEE For all $y_0  \in \mathscr{S}^M_\delta $ each  $y \in GMM(\Phi_\delta;y_0)$ is a curve of maximal slope for $\phi_\delta$
  with respect to the strong upper gradient $|\partial \phi_\delta|_{\mathcal{D}_\delta}$, in particular for all $T>0$ 
  we have the energy identity
 \begin{align}\label{slopesolution}
 \frac{1}{2} \int_0^T |y'|_{\mathcal{D}_\delta}^2(t)\,dt + \frac{1}{2} \int_0^T |\partial \phi_\delta|^2_{\mathcal{D}_\delta}(y(t))\,dt + \phi_\delta(y(T)) = \phi_\delta(y_0).
 \end{align}
  
  (iii)   (Relation to PDE) \EEE  For all $y_0  \in \mathscr{S}^M_\delta $ each  $y \in GMM(\Phi_\delta;y_0)$  is a weak  solution   of the partial differential equations of nonlinear viscoelasticity \eqref{nonlinear equation}  in the sense of  \eqref{nonlinear equation2}. \EEE

\end{theorem}

For the  linearized model we obtain the following   results.

\begin{theorem}[Solutions to the linear problem]\label{maintheorem2}
The limiting linear problem has the following properties.

(i)  (Existence/Uniqueness of MM) \EEE For all $u_0 \in H_0^1(\Omega)$ there exists a unique $u \in MM(\bar{\Phi}_0;u_0)$.

(ii)  (Curves of maximal slope) \EEE For all $u_0  \in H_0^1(\Omega)$ the minimizing movement $u \in  MM(\bar{\Phi}_0;u_0)$ is the unique curve of maximal slope for $\bar{\phi}_0$   with respect to the strong upper gradient $|\partial \bar{\phi}_0|_{\bar{\mathcal{D}}_0}$.
  
  (iii)  (Relation to PDE) \EEE For all $u_0  \in H_0^1(\Omega)$ the unique $u \in MM(\Phi_\delta;u_0)$  is a weak solution of the partial differential equations of linear viscoelasticity \eqref{linear equation}.

\end{theorem}

In contrast to Theorem \ref{maintheorem1}, we get that the weak solution to \eqref{linear equation} for given initial value $u_0 \in H^1_0(\Omega)$ is uniquely determined and a minimizing movement (and not simply a generalized one). Finally, we study the relation of the solutions to the equations \eqref{nonlinear equation} and \eqref{linear equation}.

\begin{theorem}[Relation between nonlinear and linear problems]\label{maintheorem3}
Fix a null sequence $(\delta_k)_k$ and a sequence of initial data $(y_0^k)_{k\in \N} \subset W^{2,p}_\id(\Omega)$ such that 
$$\sup\nolimits_{k\in\N} \phi_{\delta_k}(y_0^k)<\infty,  \ \ \ \ \delta_k^{-p\alpha}\int_\Omega P(\nabla^2 y_0^k) \to 0, \ \  \ \ \delta_k^{-1}(y^k_0 - \id) \to u_0 \in H_0^1(\Omega).$$ 
 Let $u$ be the unique element of $MM(\bar{\Phi}_0;u_0)$. Then the following holds: 

(i)  (Convergence of discrete solutions) \EEE For all $\tau>0$ and all discrete solutions \MK $\tilde{Y}_\tau^{\delta_k}$ \EEE as in \eqref{ds} below there is a discrete solution \MK  $\tilde{U}^0_\tau$ \EEE for the linearized system such that \MK  $\delta_k^{-1}(\tilde{Y}_\tau^{\delta_k}(t) -\id) \to \tilde{U}^0_\tau(t)$ strongly in $H^1(\Omega)$ for all $t \in [0,\infty)$.  \EEE 

(ii)  (Convergence of continuous solutions) \EEE Each sequence $y_k \in GMM(\Phi_{\delta_k};y_0^k)$, $k \in \N$, satisfies $\delta_k^{-1}(y_k(t) -\id) \to u(t)$ strongly in $H^1(\Omega)$ for all $t \in [0,\infty)$. 

(iii)  (Convergence at specific scales) \EEE For each null sequence $(\tau_k)_k$ and each sequence of discrete solutions \MK $\tilde{Y}_{\tau_k}^{\delta_k}$ as in \eqref{ds} we have $\delta_k^{-1}(\tilde{Y}_{\tau_k}^{\delta_k}(t) -\id) \to u(t)$ strongly in $H^1(\Omega)$ for all $t \in [0,\infty)$. \EEE

\end{theorem}

We remark that, in  the formulation of \BBB \cite{Braides, BCGS}, \EEE property (iii) states that  the configuration $u$ is a minimizing movement along $\phi_{\delta_k}$ at scale $\tau_k$.  Let us emphasize that the converge in Theorem \ref{maintheorem3} is with respect to the strong $H^1(\Omega)$-topology. From now on we set $f \equiv 0 $ for convenience. The general case indeed follows with minor modifications, which are standard.\EEE

\vspace{0.3cm}

\begin{figure}[H]
\centering
\begin{tikzpicture}
  \matrix (m) [matrix of math nodes,row sep=8em,column sep=8em,minimum width=6em]
  {
    \delta_k^{-1}(\tilde{Y}_{\tau_n}^{\delta_k}(t) -\id)     & \delta_k^{-1}(y_k -\id) \\
     \tilde{U}^0_{\tau_n} & u \\};
  \path[-stealth]
    (m-1-1) edge node [left] {$k \to \infty$} (m-2-1)
            edge   node [below] {$n \to \infty$} (m-1-2)
    (m-2-1) edge node [below] {$n \to \infty$}               (m-2-2)
                (m-1-2) edge node [right] {$k \to \infty$} (m-2-2) 
                (m-1-1)      
            edge node [below] {\hspace{-1.5cm}$n,k\to \infty$} (m-2-2)  ;
\end{tikzpicture}

\caption{ Illustration of the commutativity result given in Theorem \ref{maintheorem1}-Theorem \ref{maintheorem3}. The horizontal arrows are addressed in  Theorem \ref{maintheorem1} and Theorem \ref{maintheorem2}, respectively. For the vertical and diagonal arrows we refer to Theorem \ref{maintheorem3}.\EEE }\label{diagram}
 
\end{figure}

%We close this section with a remark about the boundary conditions.
%
%\begin{rem}\label{rem: part-bdy}
%{\normalfont
%
%
%}
%\end{rem}
% 

\section{Preliminaries: Generalized minimizing movements and curves of maximal slope}\label{sec3}

  In this section we first recall the relevant definitions and also give a convergence result for discrete solutions to  curves of maximal slope proved in \cite{AGS}. In Section \ref{sec: auxi-proofs} we then present a result about the limit of sequences of curves of maximal slope being a variant of results presented in \cite{CG, S2}. 

\subsection{Definitions}\label{sec: defs}

We consider a   complete metric space $(\mathscr{S},\mathcal{D})$. We say a curve $u: (a,b) \to \mathscr{S}$ is \emph{absolutely continuous} with respect to $\mathcal{D}$ if there exists $m \in L^1(a,b)$ such that
\begin{align}\label{metric-deriv}
\mathcal{D}(u(s),u(t)) \le \int_s^t m(r) \, dr \ \ \   \text{for all} \ a \le s \le t \le b.
\end{align}
The smallest function $m$ with this property, denoted by $|u'|_{\mathcal{D}}$, is called \emph{metric derivative} of  $u$  and satisfies  for a.e. $t \in (a,b)$   (see \cite[Theorem 1.1.2]{AGS} for the existence proof)
$$|u'|_{\mathcal{D}}(t) := \lim_{s \to t} \frac{\mathcal{D}(u(s),u(t))}{|s-t|}.$$
We now define the notion of a \emph{curve of maximal slope}. We only give the basic definition here and refer to \cite[Section 1.2, 1.3]{AGS} for motivations and more details.  By  $h^+:=\max(h,0)$ we denote the positive part of a function  $h$.

\begin{definition}[Upper gradients, slopes, curves of maximal slope]\label{main def2} 
 We consider a   complete metric space $(\mathscr{S},\mathcal{D})$ with a functional $\phi: \mathscr{S} \to (-\infty,+\infty]$.

(i) A function $g: \mathscr{S} \to [0,\infty]$ is called a strong upper gradient for $\phi$ if for every absolutely continuous curve $v: (a,b) \to \mathscr{S}$ the function $g \circ v$ is Borel and 
$$|\phi(v(t)) - \phi(v(s))| \le \int_s^t g(v(r)) |v'|_{\mathcal{D}}(r)\,dr \  \ \  \text{for all} \ a< s \le t < b.$$

(ii) For each $u \in \mathscr{S}$ the local slope of $\phi$ at $u$ is defined by 
$$|\partial \phi|_{\mathcal{D}}(u): = \limsup_{w \to u} \frac{(\phi(u) - \phi(w))^+}{\mathcal{D}(u,w)}.$$

(iii) An absolutely continuous curve $u: (a,b) \to \mathscr{S}$ is called a curve of maximal slope for $\phi$ with respect to the strong upper gradient $g$ if for a.e. $t \in (a,b)$
$$\frac{\rm d}{ {\rm d} t} \phi(u(t)) \le - \frac{1}{2}|u'|^2_{\mathcal{D}}(t) - \frac{1}{2}g^2(u(t)).$$
\end{definition}

 We now introduce minimizing movements. In the following we will use an approximation scheme solving suitable time-incremental minimization problems: Consider a fixed time step $\tau >0$ and suppose that an initial datum $U^0_\tau$ is given. Whenever, $U_\tau^0, \ldots, U^{n-1}_\tau$ are known, $U^n_\tau$ is defined as (if existent)
\begin{align}\label{incremental}
U_\tau^n = {\rm argmin}_{v \in \mathcal{S}} \Phi(\tau,U^{n-1}_\tau; v), \ \ \ \Phi(\tau,u; v):=  \frac{1}{2\tau} \mathcal{D}(v,u)^2 + \phi(v). 
\end{align}
Supposing that for a choice of $\tau$ a sequence $(U_\tau^n)_{n \in \N}$ solving  \eqref{incremental} exists, we define the  piecewise constant interpolation by
\begin{align}\label{ds}
\MK \tilde{U}_\tau(0) = U^0_\tau, \ \ \ \tilde{U}_\tau(t) = U^n_\tau  \ \text{for} \ t \in ( (n-1)\tau,n\tau], \ n\ge 1. \EEE
\end{align}
In the following, \MK $\tilde{U}_\tau$ \EEE will be called a \emph{discrete solution}. Note that the existence of discrete solutions is usually guaranteed by the direct method of the calculus of variations under suitable compactness, coercivity, and lower semicontinuity assumptions.  Finally, we introduce the \emph{modulus of the derivative}  
\begin{align*}
|{\tilde U'_{\tau}}|_{\mathcal{D}}(t) = \frac{\mathcal{D}(U_{\tau}^n, U_{\tau}^{n-1})}{\tau} \ \text{ for } t \in ( (n-1)\tau, n\tau], \ n\ge 1.
\end{align*}
%\label{ds_derivative}
\begin{definition}[Minimizing movements]\label{main def1}
(i) We say a curve $u: [0,\infty) \to \mathscr{S}$ is a  minimizing movement for $\Phi$ as defined in \eqref{incremental}, starting from the initial datum $u_0 \in \mathscr{S}$, if for every  sequence of timesteps $(\tau_k)_k$ with $\tau_k \to 0$ there exist discrete solutions defined in \eqref{ds} such that
\begin{align}\label{MM}
\begin{split}
&\lim_{k\to \infty} \phi(U^0_{\tau_k}) = \phi(u_0), \ \ \ \ \limsup\nolimits_{k \to \infty} \mathcal{D}(U^0_{\tau_k},u_0) < \infty, \\& \lim_{k\to \infty} \MK \mathcal{D}(\tilde{U}_{\tau_k}(t),u(t)) = 0 \ \ \  \forall t \in [0,\infty). \EEE
\end{split}
\end{align} 
By $MM(\Phi;u_0)$ we denote the collection of all minimizing movements for $\Phi$ starting from $u_0$.

(ii) Likewise,  we say a curve $u: [0,\infty) \to \mathscr{S}$ is a generalized  minimizing movement for $\Phi$ starting from  $u_0 \in \mathscr{S}$ if there exists a sequence of timesteps $(\tau_k)_k$ with $\tau_k \to 0$ and corresponding discrete solutions such that \eqref{MM} holds. The collection of all such curves is denoted by $GMM(\Phi;u_0)$.
\end{definition}

\subsection{Compactness of discrete solutions and convergence to curves of maximal slope}\label{sec: AGS-results}

Suppose again that $(\mathscr{S},\mathcal{D})$ is a complete metric space. As discussed in \cite[Remark 2.0.5]{AGS}, it is convenient to introduce a weaker topology on $\mathscr{S}$ to have more flexibility in the derivation of compactness properties. Assume that there is a Hausdorff topology $\sigma$ on $\mathscr{S}$, which is compatible with $\mathcal{D}$ in the sense that $\sigma$ is weaker than the topology induced by $\mathcal{D}$ and satisfies
\begin{align}\label{compatibility2}
u_n \stackrel{\sigma}{\to} u, \ \  v_n \stackrel{\sigma}{\to} v \ \ \ \Rightarrow \ \ \ \liminf_{n \to \infty} \mathcal{D}(u_n,v_n) \ge  \mathcal{D}(u,v).
\end{align}
Consider a functional $\phi: \mathscr{S} \to [0,+\infty)$ with the following properties:
\begin{align}\label{basic assumptions}
\begin{split}
(i) & \ \ \text{$u_n \stackrel{\sigma}{\to} u$, \ \  $\sup\nolimits_{n,m}\mathcal{D}(u_n,u_m)< \infty \  \ \Rightarrow \ \  \liminf_{n \to \infty}\phi(u_n) \ge \phi(u)$,} \\ 
(ii)& \ \ \text{for all $N \in\N$ there is a $\sigma$-sequentially compact set $K_N$ such that} \\
& \ \ \text{$\lbrace u \in \mathscr{S}: \phi(u) + \mathcal{D}(u,u_*) \le N \rbrace \subset K_N$ for some point $u_* \in \mathscr{S}$.}
\end{split}
\end{align}
Note that nonnegativity of $\phi$ can be generalized to a suitable \emph{coerciveness} condition, see \cite[(2.1.2b)]{AGS}, which we do not include here for the sake of simplicity. From \cite[Proposition 2.2.3, Theorem 2.3.3, Remark 2.3.4(i)]{AGS} we obtain the following compactness and convergence result.

\begin{theorem}\label{th: auxiliary1}
Suppose that $\phi$ satisfies \eqref{basic assumptions} and $v \in \mathscr{S} \mapsto |\partial \phi|_{\mathcal D}(v)$ is a strong upper gradient for $\phi$ and $\sigma$-lower semicontinuous. Then the following holds:

(i) Suppose that there is a sequence of initial data $(U^0_{\tau_k})_{k \in \N}$ and $u_0 \in \mathscr{S}$ with $\sup_k \mathcal{D}(U^0_{\tau_k},u_0)<+\infty$, $U^0_{\tau_k} \stackrel{\sigma}{\to} u_0$, and $\phi(U^0_{\tau_k}) \to \phi(u_0)$. Then there is an absolutely continuous curve $u:[0,\infty) \to \mathscr{S}$ and a  subsequence, \MK not relabeled, \EEE of $(\tau_k)_{k \in \N}$ such that a sequence of discrete solutions \MK $(\tilde{U}_{\tau_k})_{k \in \N}$ \EEE defined in  \eqref{ds} satisfies \MK $\tilde{U}_{\tau_k}(t) \stackrel{\sigma}{\to} u(t)$ for all $t \in [0,\infty)$.\EEE

(ii) Every $u \in GMM(\Phi;u_0)$ for each $u_0 \in \mathscr{S}$ is a curve of maximal slope for $\phi$ with respect to $|\partial \phi|_{\mathcal{D}}$ and in particular $u$ satisfies the energy identity 
\begin{align}\label{maximalslope}
\frac{1}{2} \int_0^T |u'|_{\mathcal{D}}^2(t) \, dt + \frac{1}{2} \int_0^T |\partial \phi|_{\mathcal{D}}^2(u(t)) \, dt + \phi(u(T)) = \phi(u_0) \ \  \forall T>0. 
\end{align}
Moreover, for a sequence of discrete solutions \MK $(\tilde{U}_{\tau_k})_{k \in \N}$ \EEE as in (i)  we have
\begin{align*}
\begin{split}
&\lim_{k \to \infty} \MK \phi(\tilde{U}_{\tau_k}(t)) = \phi(u(t))\EEE \ \ \ \forall t \in [0,\infty),\\
& \lim_{k \to \infty} |\partial \phi|_{\mathcal{D}}({\tilde U_{\tau_k}}) = |\partial \phi|_{\mathcal{D}}(u)\ \ \text{in} \ \ L^2_{\rm loc}([0,\infty)),\\
&   \lim_{k \to \infty} |{\tilde U'_{\tau_k}}|_{\mathcal{D}} = |u'|_{\mathcal{D}} \ \ \text{in} \ \ L^2_{\rm loc}([0,\infty)).
\end{split}
\end{align*}
\end{theorem}

In particular, Theorem \ref{th: auxiliary1}(i) states that the limit $u$ is a generalized minimizing movement, provided that $\sigma$ coincides with the topology induced by $\mathcal{D}$. We remark that $GMM(\Phi;u_0)$ could also be defined with respect to the weaker topology $\sigma$, see \cite[Definition 2.0.6]{AGS}. For our purposes, however, a definition in terms of $\mathcal{D}$ is more convenient. 

The result can be considerably improved if $\Phi$ satisfies suitable convexity properties (see \cite[Theorem 4.0.4 and Theorem 4.0.7]{AGS}).

\begin{theorem}\label{th: auxiliary2}
Suppose that $\phi$ is $\mathcal{D}$-lower semicontinuous and $\phi \ge 0$. Moreover, assume that for all $\tau>0$  and for all $w,v_0,v_1 \in \mathscr{S}$ there exists a curve $(\gamma_t)_{t \in [0,1]} \subset \mathscr{S}$ with $\gamma_0 = v_0$ and $\gamma_1 = v_1$ such that
$$\Phi(\tau,w;\gamma_t) \le (1-t)\Phi(\tau,w;   v_0  ) + t\Phi(\tau,w;  v_1  ) - \frac{t(1-t)}{2\tau} \mathcal{D}(v_0,v_1)^2 \ \ \ \forall t \in [0,1].$$
Then for each $u_0 \in \mathscr{S}$ there exists a unique $u \in MM(\Phi;u_0)$. Moreover, the assertion of Theorem \ref{th: auxiliary1} (with $\sigma$ being the topology induced by $\mathcal{D}$) holds and for a discrete solution  $\tilde{U}_{\tau}$ with $U^0_\tau = u_0$ we have 
 $\mathcal{D}(\tilde{U}_\tau(t),u(t))^2 \le \frac{1}{2}\tau^2|\partial \phi|_{\mathcal{D}}^2(u_0)$ for all $t>0$.
\end{theorem} 
Note that in contrast to Theorem \ref{th: auxiliary1}, Theorem \ref{th: auxiliary2} yields also a  uniqueness result for minimizing movements. Observe that \eqref{basic assumptions}(ii) is not necessary for Theorem \ref{th: auxiliary2} since the solvability of the problem   ${\rm argmin}_{v \in \mathscr{S}} \Phi(\tau,u; v)$ for $\tau>0$ and $u \in \mathscr{S}$ (cf. \eqref{incremental}) follows from a convexity argument.  In this setting, much more refined results can be established and we refer to \cite[Section 4]{AGS} for more details.

\subsection{Limits of curves of maximal slopes}\label{sec: auxi-proofs}

We now consider a set $\mathscr{S}$ and  a sequence of  metrics $(\mathcal{D}_n)_n$ on $\mathscr{S}$ as well as a limiting metric $\mathcal{D}$. We again assume that all metric spaces are complete. Moreover, let $(\phi_n)_n$ be a sequence of functionals with $\phi_n: \mathscr{S} \to [0,\infty]$.  Suppose that there is a Hausdorff topology $\sigma$ on $\mathscr{S}$  which is weaker than the topology induced by each $\mathcal{D}_n,\mathcal{D}$ and satisfies similarly to \eqref{compatibility2}
\begin{align}\label{compatibility}
\begin{split}
u_n \stackrel{\sigma}{\to} u, &\ \  v_n \stackrel{\sigma}{\to} v  \ \ \  \Rightarrow \ \ \ \liminf_{n \to \infty} \mathcal{D}_n(u_n,v_n) \ge  \mathcal{D}(u,v).
\end{split}
\end{align}
Moreover, assume that $(\phi_n)_n $ satisfy \eqref{basic assumptions}(ii), i.e., for all $N \in\N$ there is a $\sigma$-sequentially compact set $K_N$  and  $u_* \in \mathscr{S}$ such that for all $n \in \N$
\begin{align}\label{basic assumptions2}
\lbrace u \in \mathscr{S}: \phi_n(u) + \mathcal{D}_n(u,u_*) \le N \rbrace \subset K_N.
\end{align}
To ensure the existence of limiting curves of maximal slope, we will apply the following refined version of the Arzel\`{a} Ascoli theorem.

\begin{theorem}\label{th: auxiliary3}
Let $T>0$, let metrics $\mathcal{D}_n$, $\mathcal{D}$ and functionals $(\phi_n)_n$ be given such that  \eqref{compatibility} holds with respect to the topology $\sigma$. Let $K \subset \mathscr{S}$ be a $\sigma$-sequentially compact set. Let  $u_n:[0,T]\to \mathscr{S}$ be curves such that  
\begin{align*}
u_n(t)  \in K \ \forall n \in \N, t \in [0,T], \ \ \ \  \limsup_{n \to \infty}\mathcal{D}_n(u_n(s),u_n(t)) \le \omega(s,t) \  \ \ \forall s,t \in [0,T]
\end{align*}
for a symmetric function $\omega: [0,T]^2 \to [0,\infty)$ with
$$\lim_{(s,t) \to (r,r)} \omega(s,t) = 0 \ \ \ \ \forall r\in [0,T] \setminus \mathscr{C}, $$
where $\mathscr{C}$ is an at most countable subset of $[0,T]$. Then there exists a (not relabeled) subsequence and a limiting curve $u:[0,T] \to \mathscr{S}$ such that
$$u_n(t) \stackrel{\sigma}{\to} u(t) \ \  \ \forall t\in [0,T], \  \ \ u \text{ is $\mathcal{D}$-continuous in $[0,T] \setminus \mathscr{C}$.}  $$

\end{theorem} 

%{\cred\tt We only need the result for $\mathscr{C} = \emptyset$. Maybe formulate this way since otherwise confusing?}

\Proof  We follow the proof of \cite[Proposition 3.3.1]{AGS}  with the only difference that the lower semicontinuity condition for the metric is replaced by our condition \eqref{compatibility} along the sequence of metrics. \eop

Now  consider also a limiting functional $\phi: \mathscr{S} \to [0,\infty]$. We suppose lower semicontinuity of the functionals and the slopes in the following sense: For all $u \in \mathscr{S}$ and $(u_k)_k \subset \mathscr{S}$ we have
\begin{align}\label{eq: implication}
\begin{split}
u_k \stackrel{\sigma}{\to}  u \ \ \ \Rightarrow \ \ \ \liminf_{k \to \infty} |\partial \phi_{k}|_{\mathcal{D}_{k}} (u_{k}) \ge |\partial \phi|_{\mathcal{D}} (u), \ \ \ \liminf_{k \to \infty} \phi_{k}(u_{k}) \ge \phi(u).
\end{split}
\end{align}

We now obtain the following result about limits of curves of maximal slope.
 
\begin{theorem}\label{th:abstract convergence 1}
Consider a set $\mathscr{S}$, metrics $(\mathcal{D}_n)_{n \in \N}$ and functionals $\phi_n: \mathscr{S} \to [0,\infty]$, $n \in \N$,   as well as $\mathcal{D}$ and $\phi: \mathscr{S}\to [0,\infty]$. Suppose that there is a weaker topology $\sigma$ on $\mathscr{S}$ such that  \eqref{compatibility}, \eqref{basic assumptions2}, and the implication \eqref{eq: implication} hold. Moreover, assume that  $|\partial \phi_n|_{\mathcal{D}_n}$,  $|\partial \phi|_{\mathcal{D}}$ are strong upper gradients for $\phi_n$, $\phi$ with respect to $\mathcal{D}_n$, $\mathcal{D}$, respectively.
  
Let $T>0$ and $\bar{u} \in \mathscr{S}$. For all $n \in \N$ let $u_n$ be   a  curve of maximal slope for $\phi_n$ with respect to $|\partial \phi_n|_{\mathcal{D}_n}$ such that  
\begin{align}\label{eq: abstract assumptions1}
\begin{split}
(i)& \ \ \sup_{n \in \N} \sup_{t \in [0,T]} \big( \phi_n(u_n(t))  + \mathcal{D}_n(u_n(t),\bar{u}) \big) < \infty, \\  (ii)& \ \  u_n(0) \stackrel{\sigma}{\to}\bar{u}, \ \ \ \phi_n(u_n(0)) \to \phi(\bar{u}).
\end{split}
\end{align}
Then there exists a limiting function $u: [0,T] \to \mathscr{S}$ such that up to a subsequence, \MK not relabeled, \EEE
$$u_n(t) \stackrel{\sigma}{\to} u(t), \ \ \ \ \ \phi_n(u_n(t)) \to  \phi(u(t)) \ \ \ \forall t \in [0,T]$$
as $n \to \infty$ and $u$ is a curve of maximal slope for $\phi$ with respect to $|\partial \phi|_{\mathcal{D}}$. 
\end{theorem}

 The result is an adaption of a statement in \cite{S2}  where condition \eqref{compatibility} is replaced by a lower bound condition on the metric derivatives along the sequence. We also refer to \cite{CG}, where a similar result is proved without the assumption that  the slopes are \emph{strong}   upper gradients (cf. \cite[Definition 1.2.1 and Definition 1.2.2]{AGS} for the definition of  strong and weak upper gradients), which comes at the expense that a suitable continuity condition along  $(\phi_k)_k$ for sequences $(u_k)_k$ converging with respect to the metric  has to be imposed.

\Proof From the properties  of a curve of maximal slope we have (cf. \eqref{maximalslope})
\begin{align}\label{abstract1}
\frac{1}{2} \int_0^t |u'_n|_{\mathcal{D}_n}^2(s) \, ds + \frac{1}{2} \int_0^t |\partial \phi_n|_{\mathcal{D}_n}^2(u_n(s)) \, ds + \phi_n(u_n(t)) = \phi_n(u_n(0)) 
\end{align}
for all $t \in [0,T]$. (Here, we have used that $|\partial \phi_n|_{\mathcal{D}_n}$  are strong upper gradients for $\phi_n$ with respect to $\mathcal{D}_n$.)  From \eqref{abstract1} and the equiboundedness of  $\phi_n(u_n(t))$   (see \eqref{eq: abstract assumptions1}(i))  we get
\begin{align*}
\sup_{n \in \N}  \int_0^T|u'_n|_{\mathcal{D}_n}^2(t) \, dt + \sup_{n \in \N} \int_0^T |\partial \phi_n|_{\mathcal{D}_n}^2(u_n(t)) \, dt < \infty.
\end{align*}
Consequently, there is a function $A \in L^2((0,T))$ such that $|u_n'|_{{\mathcal{D}}_{n}} \rightharpoonup A$ weakly in $L^2((0,T))$ up to a subsequence, \MK not relabeled. \EEE  In particular, this yields
\begin{align}\label{abstract2}
\limsup_{n \to \infty} {\mathcal{D}}_n(u_n(s),u_n(t)) \le  \limsup_{n \to \infty}\int_s^t|u_n'|_{\mathcal{D}_n}\le \omega(s,t):=\int_s^t A(r)\, dr  
\end{align}
for all $0 \le s \le t \le T$ by \eqref{metric-deriv}. Using \eqref{basic assumptions2}, \eqref{eq: abstract assumptions1}(i), and \eqref{abstract2}, we can apply Theorem \ref{th: auxiliary3} and obtain an absolutely continuous curve $u: [0,T] \to  \mathscr{S}$ as well as a further  subsequence \MK (not relabeled) \EEE such that $u_n(t) \stackrel{\sigma}{\to} u(t)$ for all $t \in [0,T] $.
Moreover, recalling \eqref{compatibility} we get $\mathcal{D}(u(s),u(t)) \le \int_s^t A(r)\,dr$, which gives $|u'| \le A$. By  \eqref{eq: implication} we get
\begin{align*}
\begin{split}
 |\partial \phi|_{\mathcal{D}} (u(t)) \le    \liminf_{n \to \infty} |\partial \phi_n|_{\mathcal{D}_n} (u_n(t)),\ \ \ 
 \phi(u(t)) \le   \liminf_{n \to \infty} \phi_n(u_n(t))
\end{split}
\end{align*}
%\label{abstract6}
for $t\in [0,T]$. This together with the fact that $|u_n'|_{{\mathcal{D}}_{n}} \rightharpoonup A$ weakly in $L^2((0,T))$  and  $|u'| \le A$ gives
\begin{align*}
\begin{split}
&\frac{1}{2} \int_0^t |u'|_{\mathcal{D}}^2(s) \, ds + \frac{1}{2} \int_0^t |\partial \phi|_{\mathcal{D}}^2(u(s)) \, ds + \phi(u(t)) \\
&\le  \frac{1}{2} \int_0^t A^2(s) \, ds + \frac{1}{2} \int_0^t \liminf_{n \to \infty}  |\partial \phi_n|_{\mathcal{D}_n}^2(u_n(s)) \, ds + \liminf_{n \to \infty} \phi_n(u_n(t))  \\
&\le \liminf_{n \to \infty} \Big( \frac{1}{2} \int_0^t |u'_n|_{\mathcal{D}_n}^2(s) \, ds + \frac{1}{2} \int_0^t |\partial \phi_n|_{\mathcal{D}_n}^2(u_n(s)) \, ds + \phi_n(u_n(t))\Big)
\end{split}
\end{align*}
%\label{abstract11}
for all $t \in [0,T]$, where in the second step we used Fatou's lemma. Using   \eqref{eq: abstract assumptions1}(ii),  \eqref{abstract1}, and $\bar{u} = u(0)$   we get 
$$\frac{1}{2} \int_0^t |u'|_{\mathcal{D}}^2(s) \, ds + \frac{1}{2} \int_0^t |\partial \phi|_{\mathcal{D}}^2(u(s)) \, ds + \phi(u(t)) \le \liminf_{n \to \infty} \phi_n(u_n(0)) = \phi(u(0)).$$
On the other hand, as $|\partial \phi|_{\mathcal{D}}$ is a strong upper gradient for $\phi$ with respect to $\mathcal{D}$, we obtain (recall Definition \ref{main def2})
\begin{align*}
\phi(u(0)) \le \phi(u(t)) + \int_0^t |\partial \phi|_{\mathcal{D}}(u(s))|u'|_{\mathcal{D}}(s) \,ds.
\end{align*}
%\label{abstract12}
Therefore, combining the previous estimates and using Young's inequality we derive
$$
|u'|_{\mathcal{D}}(t) =  |\partial \phi|_{\mathcal{D}}(u(t)), \ \ \ \phi(u(0))- \phi(u(t)) = \int_0^t |\partial \phi|_{\mathcal{D}}(u(s))|u'|_{\mathcal{D}}(s) \,ds
$$
for   a.e. $t\in[0,T]$ and $\lim_{n \to \infty}\phi_n(u_n(t)) = \phi(u(t))$ for all $t \in [0,T]$. It follows that $u$ is absolutely continuous and for a.e. $t \in [0,T]$ we have
\begin{align*}
\frac{\rm d}{ {\rm d} t} \phi(u(t)) = -  |\partial \phi|_{\mathcal{D}}(u(t))|u'|_{\mathcal{D}}(t).
\end{align*}
%\label{abstract7}
This concludes the proof. \eop

We now study discrete solutions along the sequence of functionals $(\phi_n)_n$.

\begin{theorem}\label{th:abstract convergence 2}
Consider a set $\mathscr{S}$, metrics $(\mathcal{D}_n)_{n \in \N}$ and functionals $\phi_n: \mathscr{S} \to [0,\infty)$, $n \in \N$,   as well as $\mathcal{D}$ and $\phi: \mathscr{S}\to [0,\infty)$. Suppose that there is a weaker topology $\sigma$ on $\mathscr{S}$ such that  \eqref{compatibility}, \eqref{basic assumptions2} and the implication \eqref{eq: implication} hold.  Moreover, assume that    $|\partial \phi|_{\mathcal{D}}$ is a  strong upper gradient for $ \phi $ with respect to  $\mathcal{D}$.
  
Let $T>0$. Consider a  null sequence $(\tau_k)_k$ and initial data $(U^0_{\tau_k})_k$, $\bar{u}$ with
\begin{align*}
\sup\nolimits_k \mathcal{D}_k(U^0_{\tau_k},\bar{u}) < + \infty, \ \ \ \ \  U^0_{\tau_k} \stackrel{\sigma}{\to} \bar{u} , \ \ \ \ \  \phi_k(U^0_{\tau_k}) \to \phi(\bar{u}).
\end{align*}
%\label{th:abstract convergence 2-assumption}
Then for each sequence of discrete solutions $(\tilde{U}_{\tau_k})_k$  starting from $(U^0_{\tau_k})_k$ there is a   curve $u$ of maximal slope for $\phi$ with respect to $|\partial \phi|_\mathcal{D}$ such that up to a subsequence, \MK not relabeled, \EEE  $\tilde{U}_{\tau_k}(t) \stackrel{\sigma}{\to} u(t)$ and $\phi_k(\tilde{U}_{\tau_k}(t)) \to \phi(u(t))$ for $t \in [0,T]$.

\end{theorem} 

\BBB
For the proof we refer to \cite[Section 2]{Ortner}. Let us also mention the recently obtained variant \cite{BCGS} where, similarly to \cite{CG},  the lower semicontinuity along the sequence $(\phi_n)_n$ (see \eqref{eq: implication}) is replaced by a continuity condition. Note that in their setting it is  not necessary to require that  $|\partial \phi|_{\mathcal{D}}$ is a  strong upper gradient. \EEE

\section{Properties of energies and dissipation distances}\label{sec:energy-dissipation}

In this section we prove several properties about the energies and dissipation distances. Let $\delta>0$, $0 < \alpha < 1$ and recall the definition of the nonlinear energy in \eqref{nonlinear energy}-\eqref{assumptions-P} as  well as \eqref{eq: assumptions-D}. We \MK recall that  \EEE  $\mathscr{S}_\delta^M = \lbrace y \in W^{2,p}_\id(\Omega): \phi_\delta(y) \le M \rbrace$. In the whole section, $C\ge 1$ and $ 0 < c \le 1$ indicate generic constants, which may vary from line to line and depend on $M$, $\Omega$, the exponent $p>d$ (see \eqref{assumptions-P}), and on the constants in \eqref{assumptions-W},  \eqref{assumptions-P}, \eqref{eq: assumptions-D}, but are always independent of the small parameter $\delta$.

\subsection{Basic properties}

We start with some properties about the Hessian of $W$ and $D$. By $\partial^2 D^2$ we denote the Hessian and by $\partial^2_{F_1^2} D^2, \partial^2_{F_2^2} D^2$ the Hessian in direction of the first or second entry of $D^2$, respectively. Moreover, we define ${\rm sym }(F) = \frac{F + F^\top}{2}$ for $F \in \R^{d \times d}$ and recall the definition of $\C_W,\C_D$ in \eqref{linear equation}. By  $\Id \in \R^{d \times d}$ we again denote the identity matrix.

\begin{lemma}[Properties of Hessian]\label{D-lin}
Let $F_1,F_2 \in \R^{d \times d}$ and  $Y \in \R^{d \times d}$ in a neighborhood of $\Id$ such that $\partial^2 D^2(Y,Y)$ exists.

(i) We have $\partial^2 D^2(Y,Y)[(F_1,F_2),(F_1,F_2)] = \partial^2_{F_1^2}D^2(Y,Y)[F_1-F_2,F_1-F_2] = \partial^2_{F_2^2}D^2(Y,Y)[F_1-F_2,F_1-F_2]$.

(ii) We have $\partial^2 D^2(\Id,\Id)[(F_1,F_2),(F_1,F_2)] = \C_D[{\rm sym}(F_1-F_2),  {\rm sym}(F_1-F_2)]$.

(iii) There is a constant $c>0$ independent of $F$ such that $\C_W[F,F] \ge c|{\rm sym}(F)|^2$, $\C_D[F,F] \ge c|{\rm sym}(F)|^2$. 
\end{lemma}

\Proof (i) Set $H= \partial^2 D^2(Y,Y)$ for brevity. By symmetry \eqref{eq: assumptions-D}(ii) we find two fourth order tensors $H_1,H_2 : \R^{d\times d} \times \R^{d\times d} \to \R$ such that $H[(F_1,F_2),(F_1,F_2)]  = H_1[F_1,F_1] + 2H_2[F_1,F_2] + H_1[F_2,F_2]$  and $H_2[F_1,F_2] = H_2[F_2,F_1]$. Note that $H_1 = \partial^2_{F_1^2}D^2(Y,Y) =\partial^2_{F_2^2} D^2(Y,Y)$.  As $D(F,F)=0$ for all $F \in GL_+(d)$, we get  $H[(F,F),(F,F)] = 0$ for all $F \in \R^{d \times d}$. Thus, we obtain $H_1[F,F] = -H_2[F,F]$ for all $F \in \R^{d \times d}$ and we compute
\begin{align*}
H_1[F_1-F_2,&F_1-F_2]\\& = - H_2[F_1-F_2,F_1-F_2] = -H_2[F_1,F_1] + 2H_2[F_1,F_2] - H_2[F_2,F_2] \\ &= H_1[F_1,F_1] + 2H_2[F_1,F_2] + H_1[F_2,F_2] = H[(F_1,F_2),(F_1,F_2)].
\end{align*} 
Property (ii) follows from  frame indifference   \eqref{eq: assumptions-D}(v) by an elementary computation.  Finally, the growth condition for $\C_W$ and $\C_D$ stated in (iii) follow from \eqref{assumptions-W}(iii) and \eqref{eq: assumptions-D}(vi), respectively. \eop

In the following, by $\id$ we again denote the identity function.

\begin{lemma}[Rigidity]\label{lemma:rigidity}
There is  constant  $C>1$ independent of $\delta$ such that for $\delta$ sufficiently small for all $y  \in \mathscr{S}_\delta^M$  we have  

\begin{itemize}
\item[(i)] $\Vert y - \id \Vert_{H^1(\Omega)} \le C\Vert \dist(\nabla y,SO(d)) \Vert_{L^2(\Omega)}$, 
\item[(ii)] $\Vert \nabla y -\Id \Vert_{L^\infty(\Omega)}\le C\delta^{\alpha}$, \ \ \ $\Vert   y -\id \Vert_{L^\infty(\Omega)}\le C\delta^{\alpha}$.
\end{itemize}
\end{lemma}
 
\Proof (i) is a typical geometric rigidity argument, see e.g. \cite{DalMasoNegriPercivale:02, FrieseckeJamesMueller:02}: By   \cite[Theorem 3.1]{FrieseckeJamesMueller:02} and Poincar\'e's inequality we find a rotation $Q \in SO(d)$ and $b \in \R^d$ such that 
\begin{align}\label{rig1}
\Vert y - (Q\cdot + b) \Vert_{H^1(\Omega)} \le C\Vert \dist(\nabla y,SO(d)) \Vert_{L^2(\Omega)}.
\end{align}
Passing to a trace estimate and using $y = \id$ on $\partial \Omega$, we get $\Vert \id - (Q\cdot + b) \Vert_{L^2(\partial \Omega)} \le C\Vert \dist(\nabla y,SO(d)) \Vert_{L^2(\Omega)}$. Using  \cite[Lemma 3.3]{DalMasoNegriPercivale:02}  we then find $|b| + |Q- \Id| \le C\Vert \id - (Q\cdot + b) \Vert_{L^2(\partial \Omega)}$ for a constant only depending on $\Omega$. This together with \eqref{rig1} implies (i).

We now prove (ii). By the definition of $\phi_\delta$ and \eqref{assumptions-P}(iii) we get $\Vert \nabla^2 y \Vert^p_{L^p(\Omega)} \le CM\delta^{p\alpha}$ for all $y \in \mathscr{S}^M_\delta$. As $p>d$, Poincar\'e's inequality yields some $F \in \R^{d \times d}$ and $b \in \R^d$ such that  
\begin{align}\label{rig2}
\Vert y -  (F\cdot + b)\Vert_{W^{1,\infty}(\Omega)} \le C\delta^{\alpha}
\end{align}
for a constant additionally depending on $\Omega$, M, and $p$. Using $\phi_\delta(y) \le M$, \eqref{assumptions-W}(iii), and (i) we compute
\begin{align*}
\Vert (F\cdot + b)- \id \Vert^2_{H^1(\Omega)} \le C\Vert \dist(\nabla y,SO(d)) \Vert^2_{L^2(\Omega)} + C|\Omega|\delta^{2\alpha}
 \le  C\delta^2M + C|\Omega|\delta^{2\alpha} . 
\end{align*}
Since $\alpha \le 1$, this gives $|b| + |F - \Id| \le C\delta^\alpha$, which together with \eqref{rig2}  yields (ii). \eop

In the following we set for shorthand $H_Y := \frac{1}{2}\partial^2_{F_1^2} D^2(Y,Y) = \frac{1}{2}\partial^2_{F_2^2} D^2(Y,Y)$  for $Y \in GL_+(d)$ and given a deformation $y \in W^{2,p}_\id(\Omega)$ we also introduce the mapping $H_{\nabla y}: \Omega \to \R^{d \times d \times d \times d}$   by $H_{\nabla y}(x) = H_{\nabla y(x)}$ for $x \in \Omega$. Recall the definition of $\mathcal{D}_\delta, \bar{\mathcal{D}}_0$ in \eqref{eq: D,D0} and $\C_W$ below \eqref{linear equation}.

\begin{lemma}[Dissipation and energy]\label{lemma: metric space-properties}
There are constants $0<c<1$, $C>1$ independent of $\delta$ such that   for all $y,y_0,y_1 \in \mathscr{S}_\delta^M$ for $\delta$ sufficiently small we have  

\begin{itemize}
\item[(i)] $\big|\delta^2\mathcal{D}_\delta(y_0,y_1)^2 - \int_\Omega H_{\nabla y_0}[\nabla (y_1 -  y_0),\nabla (y_1 -   y_0) ]| \le C \Vert \nabla (y_1-   y_0) \Vert^3_{L^3(\Omega)}$,
\item[(ii)] $c \Vert  y_1 - y_0 \Vert_{H^1(\Omega)} \le \delta\mathcal{D}_\delta(y_0,y_1) \le C\Vert  y_1 - y_0 \Vert_{H^1(\Omega)}$,
\item[(iii)] $\big|\mathcal{D}_\delta(y_0,y_1)^2 - \bar{\mathcal{D}}_0(u_0,u_1)^2\big| \le C\delta^\alpha$,
\item[(iv)] $\big|\delta^{-2} \int_\Omega W(\nabla y) - \int_\Omega \frac{1}{2}\C_W[e (u),e (u) ]\big| \le C\delta^\alpha,$
\end{itemize}
where $u = \delta^{-1}(y - \id)$ and  $u_i = \delta^{-1}(y_i - \id)$, $i=0,1$. In particular, (ii) shows that the topologies induced by $\mathcal{D}_\delta$ and $\Vert \cdot \Vert_{H^1(\Omega)}$ coincide.
\end{lemma}

\Proof
 Recall that $D^2$ is $C^3$ in a neighborhood of $(\Id,\Id)$. In view of the uniform bound on $\nabla y_0, \nabla y_1 $ (see Lemma \ref{lemma:rigidity}(ii)) and a Taylor expansion of $D^2$ at $(\nabla y_0, \nabla y_0)$, we derive by Lemma \ref{D-lin}
\begin{align*}
\begin{split}
\int_\Omega D^2(\nabla y_0,\nabla y_1) &= \int_\Omega H_{\nabla y_0}[\nabla (y_1 -  y_0),\nabla (y_1 -   y_0) ]  + O( \Vert \nabla (y_1-   y_0) \Vert^3_{L^3(\Omega)}).
\end{split}
\end{align*}
This gives (i).  We obtain $\Vert H_{\nabla y_0} - \C_D \Vert_{L^\infty(\Omega)} \le C\delta^\alpha$ by regularity of $D$ and Lemma \ref{lemma:rigidity}(ii). This together with (i), Lemma \ref{lemma:rigidity}(ii), and Lemma \ref{D-lin} yields  
\begin{align}\label{eq:NNN}
\begin{split}
\int_\Omega D^2(\nabla y_0,\nabla y_1)& = \int_\Omega \C_D[e(y_1)-e(y_0),e(y_1)-e(y_0)]\\& \ \ \  + O(\delta^\alpha \Vert \nabla y_1- \nabla y_0 \Vert^2_{L^2(\Omega)}).
\end{split}
\end{align}
Now by    \eqref{eq:NNN},  Lemma \ref{D-lin}(iii), and Korn's inequality we derive  for $\delta$ small enough
  \begin{align*}
\int_\Omega D^2(\nabla y_0,\nabla y_1) & \ge c \Vert e(y_1)-e(y_0) \Vert^2_{L^2(\Omega)} + O( \delta^\alpha\Vert \nabla y_1- \nabla y_0 \Vert^2_{L^2(\Omega)}) \\& \ge c  \Vert \nabla y_1- \nabla y_0 \Vert^2_{L^2(\Omega)}.
\end{align*}
Here we used that $y_1 - y_0 = 0$ on $\partial \Omega$. The first inequality in (ii) follows from Poincar\'e's inequality. The other inequality can be seen along similar lines. By Lemma \ref{lemma:rigidity}(i),  \eqref{assumptions-W}(iii) and the fact that $y_0,y_1 \in \mathscr{S}_\delta^M$ we get
\begin{align}\label{eq:remD}
\Vert \nabla y_i - \Id \Vert^2_{L^2(\Omega)} \le C\Vert \dist(\nabla y_i,SO(d)) \Vert^2_{L^2(\Omega)} \le C\phi_\delta(y_i)  \le CM\delta^2  
\end{align}
for $i=0,1$. Recalling the definition of $\mathcal{D}_\delta, \bar{\mathcal{D}}_0$, we now obtain (iii) by \eqref{eq:NNN}.

Finally, to see (iv), an argument very similar to (i), essentially relying on a Taylor expansion and Lemma \ref{lemma: metric space-properties}(ii), yields
$$\Big|\delta^{-2} \int_\Omega W(\nabla y) - \int_\Omega \frac{1}{2}\C_W[e (u),e (u) ]\Big| \le C\delta^{\alpha-2} \Vert \nabla y - \Id \Vert^2_{L^2(\Omega)},$$
which together with \eqref{eq:remD} implies the claim.  \eop

We close this section with proving differentiablity of  $\int_\Omega W(\nabla y)$.

\begin{lemma}[Differentiablity of $\int_\Omega W(\nabla y)$]\label{lemma:C1}
For $(y_n)_n \subset \mathscr{S}_\delta^M$ and $y \in \mathscr{S}_\delta^M$ with $\mathcal{D}_\delta(y_n,y) \to 0$, we have
$$\lim_{n \to \infty}  \frac{\int_\Omega W(\nabla y_n) - \int_\Omega W(\nabla y) - \int_\Omega \partial_FW(\nabla y) : (\nabla y_n - \nabla y)}{\mathcal{D}_\delta(y_n,y)} = 0.$$
\end{lemma}

\Proof
By a Taylor expansion we find a universal constant $C'>0$ such that  $|W(F_2) - W(F_1) - \partial_F W(F_1) : (F_2 - F_1)| \le C'|F_1 - F_2|^2$ for all $F_1,F_2$ with $|F_1 - \Id|,|F_2-\Id| \le C\delta^\alpha$, where  $C$ is the constant in Lemma \ref{lemma:rigidity}(ii). This together with Lemma \ref{lemma:rigidity}(ii) and Lemma \ref{lemma: metric space-properties}(ii) gives the result. \eop

\subsection{Metric spaces and convexity}\label{sec: metric}

In this section we show that $(\mathscr{S}^M_\delta, \mathcal{D}_\delta)$, $(H^1_0(\Omega),\bar{\mathcal{D}}_0)$ are complete metric spaces and derive convexity properties for the energies and dissipation distances. 

\begin{theorem}[Properties of $(\mathscr{S}^M_\delta, \mathcal{D}_\delta)$ and $\phi_\delta$]\label{th: metric space}
For $\delta>0$ small enough we have
\begin{itemize}
\item[(i)] $(\mathscr{S}^M_\delta, \mathcal{D}_\delta)$ is a complete metric space.
\item[(ii)] Compactness: If $(y_n)_n \subset \mathscr{S}^M_\delta$, then $(y_n)_n$ admits a subsequence converging weakly in $W^{2,p}(\Omega)$, strongly in  $W^{1,\infty}(\Omega)$, and with respect to $\mathcal{D}_\delta$.
\item[(iii)] Lower semicontinuity: $\mathcal{D}_\delta(y_n,y) \to 0$ \   \ $\Rightarrow$   \ \ $\liminf_{n \to \infty} \phi_\delta(y_n) \ge \phi_\delta(y)$.
\end{itemize}
\end{theorem}

\Proof   First, recalling \eqref{nonlinear energy} and \eqref{assumptions-P}(iii), we have $\Vert \nabla^2 y \Vert^p_{L^p(\Omega)} \le CM\delta^{p\alpha}$ for all $y \in \mathscr{S}^M_\delta$, which together with Lemma \ref{lemma:rigidity}(ii)   shows  $\sup_{y \in \mathscr{S}^M_\delta}\Vert y \Vert_{W^{2,p}(\Omega)} < \infty$. This implies (ii) recalling $p>d$ and also using Lemma \ref{lemma: metric space-properties}(ii).  In particular, for a sequence $(y_n)_n$ converging to $y$ with respect to $\mathcal{D}_\delta$ we have   $y_n \rightharpoonup y$ weakly in $W^{2,p}(\Omega)$ and $y_n \to y$ strongly in $W^{1,\infty}(\Omega)$. Then (iii) follows from   Fatou's lemma    and the fact that $\liminf_{n\to \infty} \int_\Omega P(\nabla^2 y_n) \ge \int_\Omega P(\nabla^2 y)$ by   \eqref{assumptions-P}(ii).

We now finally show (i). Apart from the positivity, all properties of a metric follow directly from \eqref{eq: assumptions-D} and \eqref{eq: D,D0}. To show that if $\mathcal{D}_\delta(y_0,y_1)=0$ for  $y_0, y_1 \in \mathscr{S}^M_\delta$, then $y_0=y_1$, we apply Lemma \ref{lemma: metric space-properties}(ii). Finally, it remains to show that $(\mathscr{S}^M_\delta, \mathcal{D}_\delta)$ is complete. Let $(y_k)_k$ be a Cauchy sequence with respect to $\mathcal{D}_\delta$. By  (ii) we find $y \in W^{2,p}(\Omega)$ and a subsequence \MK (not relabeled) \EEE such that $y_k \to y$ in $W^{1,\infty}(\Omega)$. Then also $\lim_{k\to \infty}\mathcal{D}_\delta(y_k,y) = 0$ by Lemma \ref{lemma: metric space-properties}(ii). By (iii) we get $y \in \mathscr{S}_\delta^M$. The fact that $(y_k)_k$ is a Cauchy sequence now implies that the whole sequence $y_k$ converges to $y$ with respect to $\mathcal{D}_\delta$. This concludes the proof. \eop

Similar properties can be derived in the linear setting. Recall the definition of $\bar{\mathcal{D}}_0$ in \eqref{eq: D,D0}.

\begin{theorem}[Properties of $(H^1_0(\Omega), \bar{\mathcal{D}}_0)$ and $\bar{\phi}_0$]\label{th: metric space-lin}
We have

\begin{itemize}
\item[(i)] $(H^1_0(\Omega), \bar{\mathcal{D}}_0)$ is a complete metric space.
\item[(ii)] Continuity: $\bar{\mathcal{D}}_0(u_n, u) \to 0$ \ \   $\Rightarrow$ \ \  $\lim_{n \to \infty} \bar{\phi}_0(u_n) = \bar{\phi}_0(u)$.
\end{itemize}
\end{theorem}

\Proof By Lemma \ref{D-lin}(iii) we find a constant $c>0$ such that 
$$\bar{\mathcal{D}}_0(u_0,u_1)^2 \ge c \Vert e(u_0) - e(u_1) \Vert_{L^2(\Omega)}^2 \ge \Vert u_0 - u_1 \Vert^2_{H^1(\Omega)},  $$ 
where the last step follows from Korn's and Poincare's inequality. This show that $(H^1_0(\Omega), \bar{\mathcal{D}}_0)$ is a complete metric space, where $\bar{\mathcal{D}}_0$ is equivalent to the metric induced by $\Vert \cdot \Vert_{H^1(\Omega)}$. Recalling \eqref{linear energy} we find that $\bar{\phi}_0$ is  continuous with respect to $\bar{\mathcal{D}}_0$. \eop

The following  properties are crucial to use the theory in \cite{AGS}. 

\begin{theorem}[Convexity and generalized geodesics in the nonlinear setting]\label{th: convexity}
There is a constant $C \ge 1$ independent of $\delta$ such that for $\delta$ small and for all $y_0,y_1 \in \mathscr{S}^M_\delta$:  
\begin{align*}
(i)& \ \    \mathcal{D}_\delta(y_s,y_0)^2  \le  s^2\mathcal{D}_\delta(y_1,y_0)^2  (1 + C \Vert \nabla y_1 - \nabla y_0 \Vert_{L^\infty(\Omega)}), \\
(ii) & \ \ \phi_\delta(y_s)  \le (1-s) \phi_\delta(y_0) + s\phi_\delta(y_1),
\end{align*}
where $y_s := (1-s) y_0 + sy_1$, $s \in [0,1]$.

\end{theorem}

Note that $y_s$ is not a geodesic in the sense of \cite[Definition 2.4.2]{AGS}, but $y_s$ can be understood as a generalized geodesic. We also refer to \cite[Section 3.2, Section 3.4]{MOS} for a discussion about generalized geodesics in a related setting.

\Proof Let $y_s = (1-s)y_0 + sy_1$. By Lemma \ref{lemma: metric space-properties}(i) we obtain
\begin{align*}
\delta^2\mathcal{D}_\delta(y_1,y_0)^2 &\ge \int_\Omega H_{\nabla y_0}[\nabla (y_1-y_0),\nabla (y_1-y_0)] - C\int_\Omega|\nabla y_1 - \nabla y_0|^3.  
\end{align*}
Likewise, we get 
   \begin{align*}
\delta^2\mathcal{D}_\delta(y_s,y_0)^2 & \le s^2\int_\Omega H_{\nabla y_0}[\nabla (y_1-y_0),\nabla (y_1-y_0)]+ Cs^3\int_\Omega|\nabla y_1 - \nabla y_0|^3.
\end{align*}
Combining the  two estimates, we therefore obtain 
  \begin{align*}
\mathcal{D}_\delta(y_s,y_0)^2 & \le s^2\big( \mathcal{D}_\delta(y_1,y_0)^2 + C\delta^{-2}\Vert \nabla y_1 - \nabla y_0 \Vert_{L^3(\Omega)}^3\big),
\end{align*}
which together with Lemma \ref{lemma: metric space-properties}(ii) shows (i). To see (ii),   it suffices to show  $\int_\Omega W(\nabla y_s) \le (1-s)\int_\Omega W(\nabla y_0) + s \int_\Omega W(\nabla y_1)$ since $P$ is convex (see \eqref{assumptions-P}(ii)). A Taylor expansion gives $\int_\Omega W(\nabla y) = \frac{1}{2}\int_\Omega \C_W[\nabla y,\nabla y] + \omega(\nabla y)$ for a (regular) function $\omega:  \R^{d \times d} \to \R$  with $\partial_F \omega(0) = 0$ and  $\partial^2_{F^2} \omega(0) =0$. We get 
\begin{align}\label{quadratic convexity}
\begin{split}
 \int_\Omega \C_W[\nabla y_s,\nabla y_s]    &= (1-s)  \int_\Omega \C_W[\nabla y_0,\nabla y_0] + s    \int_\Omega \C_W[\nabla y_1,\nabla y_1]  \\ & \ \ \     - s(1-s)   \int_\Omega \C_W[\nabla (y_1 - y_0),\nabla (y_1 - y_0)]. 
 \end{split}
 \end{align}
Denote by  $B_{2C\delta^\alpha}(\Id) \subset \R^{d \times d}$  the ball with center $\Id$ and radius $2C\delta^\alpha$ with the constant $C$ from Lemma \ref{lemma:rigidity}(ii). Since $F \mapsto \omega(F) + \frac{1}{2}\Vert \partial^2_{F^2} \omega \Vert_{L^\infty(B_{2C\delta^\alpha}(\Id))}|F|^2$ is convex on $B_{2C\delta^\alpha}(\Id)$, we get by  Lemma \ref{lemma:rigidity}(ii)
   \begin{align*}
\int_\Omega \omega(\nabla y_s ) & \le s\int_\Omega \omega (\nabla y_0) + (1-s) \int_\Omega \omega(\nabla y_1 ) \\& \ \ \ + \frac{1}{2}s(1-s) \Vert \partial^2_{F^2} \omega\Vert_{L^\infty(B_{2C\delta^\alpha}(\Id))} \int_\Omega|\nabla y_1 - \nabla y_0|^2.\notag
\end{align*}
By    the fact that $\partial^2_{F^2 }w(0) =0$ and the regularity of $\omega$  we find $\Vert \partial^2_{F^2 } \omega \Vert_{L^\infty(B_{2C\delta^\alpha}(\Id))} \le C\delta^\alpha$.  Combining the previous  three  estimates and recalling that  $\int_\Omega W(\nabla y) = \frac{1}{2}\int_\Omega \C_W[\nabla y,\nabla y] + \omega(\nabla y)$,  we conclude 
   \begin{align*}
&\int_\Omega W(\nabla y_s) - (1-s)\int_\Omega W(\nabla y_0) - s\int_\Omega W(\nabla y_1) \\ & \le  - s(1-s)  \int_\Omega \C_W[\nabla (y_1 - y_0),\nabla (y_1 - y_0)] +\frac{1}{2}s(1-s) C\delta^\alpha  \int_\Omega|\nabla (y_1 -  y_0)|^2 \le 0
\end{align*}
for $\delta$ small enough, where the last step follows from Lemma \ref{D-lin}(iii) and Korn's inequality. \eop

We note without proof that by a similar reasoning as in (ii) one can show that  for given  $w \in \mathscr{S}^M_\delta$ 
$$\mathcal{D}_\delta(y_s,w)^2  \le  (1-s)\mathcal{D}_\delta(y_0,w)^2 + s\mathcal{D}_\delta(y_1,w)^2 - s(1-s)(1 - C\delta^\alpha)\mathcal{D}_\delta(y_1,y_0)^2.$$
This implies that $\mathcal{D}_\delta$ is $2(1-C\delta^\alpha)$-convex in the sense of \cite[Assumption 4.0.1]{AGS}. Note that this property is not strong enough to apply directly the results in \cite[Section 2.4, Section 4]{AGS}. \BBB Nevertheless, we will be able to derive representations and lower semicontinuity properties  for the slopes  by direct computations (see Lemma \ref{lemma: slopes}, Lemma \ref{lemma: lsc-slope} below.) \EEE However, in the linear setting we obtain $2$-convexity as the following result shows.

\begin{lemma}[Convexity in the linear setting]\label{lemma: convexity2}
For all $u_0,u_1 \in H^1_0(\Omega)$ and $v \in H^1_0(\Omega)$ with $u_s := (1-s) u_0 + su_1$ we have
$$\bar{\mathcal{D}}_0(u_s,v)^2  \le  (1-s) \bar{\mathcal{D}}_0(u_0,v)^2 + s \bar{\mathcal{D}}_0(u_1,v)^2 - s(1-s) \bar{\mathcal{D}}_0(u_1,u_0)^2.$$
\end{lemma}

\Proof The property follows from an elementary computation as in \eqref{quadratic convexity} taking into account that   $\bar{\mathcal{D}}_0^2$ is quadratic. \eop

\subsection{Properties of local slopes}\label{sec: slopes}

We now derive representations and properties of the slopes corresponding to $\phi_\delta$ and $\bar{\phi}_0$. Recall Definition \ref{main def2}.

\begin{lemma}[Slopes]\label{lemma: slopes}
(i) For $\delta>0$ small enough the local slopes in the nonlinear setting admit the representation
\begin{align*}
&|\partial \phi_\delta|_{\mathcal{D}_\delta}(y) = \sup_{w \neq y} \  \frac{(\phi_\delta(y) - \phi_\delta(w))^+}{\mathcal{D}_\delta(y,w) (1 + C \Vert \nabla y - \nabla w \Vert_{L^\infty(\Omega)})^{1/2}} \ \ \ \  \forall y \in \mathscr{S}_\delta^M,
\end{align*}
 where $C$ is the constant from Theorem \ref{th: convexity}. The slopes are  lower semicontinuous with respect to both  $H^1(\Omega)$ and $\mathcal{D}_\delta$ and are strong upper gradients for $\phi_\delta$.

(ii) The local slope for the linear energy $\bar{\phi}_0$ admits the representation 
$$ |\partial \bar{\phi}_0|_{\bar{\mathcal{D}}_0}(u) = \sup_{v \neq u} \  \frac{(\bar{\phi}_0(u) - \bar{\phi}_0(v))^+}{\bar{\mathcal{D}}_0(u,v)},$$
 and is a strong upper gradient for $\bar{\phi}_0$.  
\end{lemma}

\Proof  Before we start with the actual proof, let us recall from \cite[Lemma 1.2.5]{AGS} that in a complete metric space $(\mathscr{S,\mathcal{D}})$ with energy $\phi$ one has that  $|\partial \phi|_{\mathcal{D}}$ is a weak upper gradient for $\phi$ in the sense of \cite[Definition 1.2.2]{AGS}. We do not repeat the definition of weak upper gradients, but only mention that weak upper gradients are also strong upper gradients if for each absolutely continuous curve $z:(a,b) \to \mathscr{S}$ with $|\partial \phi|_{\mathcal{D}}(z)|z'|_{\mathcal{D}} \in L^1(a,b)$, the function $\phi \circ z$ is absolutely continuous.

Moreover, \cite[Lemma 1.2.5]{AGS} also states that, if $\phi$ is $\mathcal{D}$-lower semicontinuous, then the global slope
\begin{align}\label{global slope}
\MK \mathcal{S}_{\phi}(v) \EEE := \sup_{w \neq v} \frac{(\phi(v) - \phi(w))^+}{\mathcal{D}(v,w)}
\end{align}
is a strong (and thus also weak) upper gradient for $\phi$. 
%{\cred \tt Should we put this in Section 3? However, it is only used in this proof here.}

We now give the proof of (i). We partially follow the proofs of Theorem 2.4.9 and Corollary 2.4.10 in \cite{AGS}. To confirm the representation of $|\partial \phi_\delta|_{\mathcal{D}_\delta}$, we use the definition of the local slope in Definition \ref{main def2} and obtain with $C$ being the constant from Theorem \ref{th: convexity}(i)
\begin{align*}
|\partial \phi_\delta|_{\mathcal{D}_\delta}(y) & = \limsup_{w \to y} \frac{(\phi_\delta(y) - \phi_\delta(w))^+}{\mathcal{D}_\delta(y,w)} = \limsup_{w \to y} \frac{(\phi_\delta(y) - \phi_\delta(w))^+}{\mathcal{D}_\delta(y,w) (1 + C \Vert \nabla y- \nabla w \Vert_{\infty})^{1/2}}  \\
&\le \sup_{w \neq y} \  \frac{(\phi_\delta(y) - \phi_\delta(w))^+}{\mathcal{D}_\delta(y,w) (1 + C \Vert \nabla y - \nabla w \Vert_{\infty})^{1/2}},
\end{align*}
where in the second \MK equality \EEE we used that $w \to y$ (with respect to $\mathcal{D}_\delta$) implies $\Vert \nabla w  - \nabla y\Vert_{L^ \infty(\Omega)} \to 0$ by Theorem \ref{th: metric space}(ii).   To see the other inequality, it is not restrictive to suppose that $y \neq w$ and  
\begin{align}\label{proof2.1}
\phi_\delta(y) - \phi_\delta(w)>0.
\end{align}
By  Theorem \ref{th: convexity}(ii) with $y_0 = y$ and $y_1 = w$  we get
 \begin{align*}
\frac{\phi_\delta(y) - \phi_\delta(y_s)}{\mathcal{D}_\delta(y,y_s)}  \ge \frac{\phi_\delta(y) - \phi_\delta(w)}{\mathcal{D}_\delta(y,w)}   \frac{s\mathcal{D}_\delta(y,w)}{\mathcal{D}_\delta(y,y_s)} 
 \end{align*}
for all  $s \in [0,1]$, where $y_s = (1-s)y + s w$. Then we derive by \eqref{proof2.1} and Theorem \ref{th: convexity}(i)
\begin{align*}
|\partial \phi_\delta|_{\mathcal{D}_\delta}(y) \ge  \frac{\phi_\delta(y) - \phi_\delta(w)}{\mathcal{D}_\delta(y,w) (1+ C \Vert \nabla y - \nabla w \Vert_\infty)^{1/2}}.
\end{align*}
The claim now follows by taking the supremum with respect to $w$.  To confirm  the lower semicontinuity, we consider $y_h \to y$ in $\mathcal{D}_\delta$ or equivalently in $H^1(\Omega)$ (see Lemma \ref{lemma: metric space-properties}(ii)). If $w \neq y$, then $w \neq y_h$ for $h$ large enough and thus
\begin{align*}
\liminf_{h \to \infty}  |\partial \phi_\delta|_{\mathcal{D}_\delta}(y_h) & \ge \liminf_{h \to \infty}   \frac{(\phi_\delta(y_h) - \phi_\delta(w))^+}{\mathcal{D}_\delta(y_h,w) (1 + C\Vert \nabla y_h - \nabla w \Vert_{\infty})^{1/2}} \\
& \ge \frac{(\phi_\delta(y) - \phi_\delta(w))^+}{\mathcal{D}_\delta(y,w) (1 + C\Vert \nabla y - \nabla w \Vert_{\infty})^{1/2}}, 
\end{align*}
where we   used Theorem \ref{th: metric space}(ii),(iii). By taking the supremum with respect to $w$ the lower semicontinuity   follows.  

It remains to show that $|\partial \phi_\delta|_{\mathcal{D}_\delta}$ is a strong upper gradient.  With Lemma \ref{lemma:rigidity}(ii), for $\delta$ small enough we find $\mathcal{S}_{\phi_\delta}(y) \le 2  |\partial \phi_\delta|_{\mathcal{D}_\delta}(y)$ with $\mathcal{S}_{\phi_\delta}$ as introduced in \eqref{global slope}. Recalling the remarks at the beginning of the proof, to show that $|\partial \phi_\delta|_{\mathcal{D}_\delta}$ is a strong upper gradient we have to check that for all absolutely continuous $z:(a,b) \to \mathscr{S}^M_\delta$ with $|\partial \phi_\delta|_{\mathcal{D}_\delta}(z)|z'|_{{\mathcal D}_\delta} \in L^1(a,b)$, the function $\phi_\delta \circ z$ is absolutely continuous. First, it follows $\mathcal{S}_{\phi_\delta}(z)|z'|_{{\mathcal D}_\delta} \in L^1(a,b)$ as  $\mathcal{S}_{\phi_\delta} \le 2  |\partial \phi_\delta|_{\mathcal{D}_\delta}$.  Since $\phi_\delta$ is $\mathcal{D}_\delta$-lower semicontinous, $\mathcal{S}_{\phi_\delta}$ is a strong upper gradient. Thus,  we indeed get that $\phi_\delta \circ z$ is absolutely continuous, see Definition \ref{main def2}.

We now concern ourselves with (ii). The representation of the local slope follows from the convexity property in Lemma \ref{lemma: convexity2} as was shown in \cite[Theorem 2.4.9]{AGS}. Therefore, $\mathcal{S}_{\bar{\phi}_0} =  |\partial \bar{\phi}_0|_{\bar{\mathcal{D}}_0}$, which is $\bar{\mathcal{D}}_0$ lower semicontinous by Lemma \ref{th: metric space-lin}(ii) and thus  $|\partial \bar{\phi}_0|_{\bar{\mathcal{D}}_0}$ is a strong upper gradient. \eop

\section{Proof of the main results}\label{sec results}

In this section we give the proof of Theorem \ref{maintheorem1}-Theorem \ref{maintheorem3}.
  
\subsection{Existence of curves of maximal slope}

 In this section we prove the first two parts of Theorem  \ref{maintheorem1} and Theorem \ref{maintheorem2}, which essentially follow from the properties of the metric spaces established in Section \ref{sec: metric}, \ref{sec: slopes} by applying the general results recalled in Section \ref{sec: AGS-results}. 

\begin{proof}[Proof of Theorem  \ref{maintheorem1}(i),(ii)]
First, we note  that the assumptions of Theorem \ref{th: auxiliary1} are satisfied by Lemma \ref{lemma: slopes}(i) and Lemma \ref{th: metric space}(ii),(iii), where we let $\mathscr{S} = \mathscr{S}_\delta^M$  and let  $\sigma$ be the topology induced by $\mathcal{D}_\delta$. 

(i)  Fix $y_0 \in \mathscr{S}^M_\delta$. Define the initial data  $U^0_\tau = y_0$ for all $\tau>0$. Applying Theorem \ref{th: auxiliary1}(i)  we find a curve $y$ which is the limit of a sequence of discrete solutions with $y(0) = y_0$. Thus, in view of Definition \ref{main def1}, $y \in GMM(\Phi_{\delta};y_0)$, which is therefore nonempty.  

(ii) To see that generalized minimizing movements are curves of maximal slope, it suffices to apply Theorem \ref{th: auxiliary1}(ii).
\end{proof}

\begin{proof}[Proof of Theorem  \ref{maintheorem2}(i),(ii)]
In the linear setting the convexity property given in Lemma \ref{lemma: convexity2}  holds and $\bar{\phi}_0$ is convex by \eqref{linear energy} and Lemma \ref{D-lin}(iii). Thus, Theorem \ref{th: auxiliary2} is applicable.  Apart from uniqueness, the result then follows from Theorem \ref{th: auxiliary2}. It remains to show that the unique  minimizing movement is also the unique curve of maximal slope for $\bar{\phi}_0$ with respect to the strong upper gradient $|\partial \bar{\phi}_0|_{\bar{\mathcal{D}}_0}$. To this end, we follow an idea used, e.g., in \cite{Gigli}.

We first observe that the metric derivative $|u'|^2_{\bar{\mathcal{D}}_0}$ is convex. Indeed, let $u^1,u^2:[0,\infty) \to H^1_0(\Omega)$ be two curves. We get for $u^{3} = \frac{1}{2}(u^1 + u^2)$ by Young's inequality  (define $v^i = u^{i}(s) - u^{i}(t)$, $i=1,2$, for brevity)
\begin{align*}
\bar{\mathcal{D}}_0&(u^{3}(s), u^{3}(t))^2 = \int_\Omega \C_D[e((v^1+v^2)/2), e((v^1+v^2)/2)]\\& = \sum\nolimits_{i=1,2} \frac{1}{4}\int_\Omega \C_D[e(v^i), e(v^i)] + \frac{1}{2}\int_\Omega \C_D[e(v^1), e(v^2)] \\
& \le  \sum\nolimits_{i=1,2}\frac{1}{2}\int_\Omega \C_D[e(v^i), e(v^i)] =  \frac{1}{2} \bar{\mathcal{D}}_0(u^{1}(s), u^{1}(t))^2 + \frac{1}{2} \bar{\mathcal{D}}_0(u^{2}(s), u^{2}(t))^2.
\end{align*}
Dividing by $|s-t|^2$ and letting $s$ go to $t$ we obtain the claim. We also anticipate from Lemma \ref{lemma: lin-slope} below that $u \mapsto |\partial \bar{\phi}_0|^2_{\bar{\mathcal{D}}_0}(u)$ is convex. 

Assume there were two different curves of maximal slope $u^1$, $u^2$ starting from $u_0$, i.e., we find some $T$ such that $e(u^1(T)) \neq e(u^2(T))$ since otherwise the curves would coincide by Korn's inequality. Set $u^{3} = \frac{1}{2}(u^1 + u^2)$ and compute by the strict convexity of $\C_W$ on $\R^{d \times d}_{\rm sym}$ (see Lemma \ref{D-lin}(iii)), the convexity properties of the slope and metric derivative, and \eqref{maximalslope}
\begin{align*}
\bar{\phi}_0(u_0) &= \frac{1}{2}\sum_{i=1,2} \Big( \frac{1}{2} \int_0^T |(u^i)'|_{\bar{\mathcal{D}}_0}^2(t) \, dt + \frac{1}{2} \int_0^T |\partial \bar{\phi}_0|_{\bar{\mathcal{D}}_0}^2(u^i(t)) \, dt + \bar{\phi}_0(u^i(T)) \Big)\\
& >  \frac{1}{2} \int_0^T |(u^{3})'|_{\bar{\mathcal{D}}_0}^2(t) \, dt + \frac{1}{2} \int_0^T |\partial \bar{\phi}_0|_{\bar{\mathcal{D}}_0}^2(u^{3}(t)) \, dt + \bar{\phi}_0(u^{3}(T)),
\end{align*}
which contradicts the fact that $ |\partial \bar{\phi}_0|_{\bar{\mathcal{D}}_0}$ is an upper gradient (see Definition \ref{main def2}(i) and use Young's inequality). This contradiction establishes uniqueness and concludes the proof. 
\end{proof}

\subsection{$\Gamma$-convergence and lower semicontinuity}

As a preparation for the passage to the linear problem, we recall and prove $\Gamma$-convergence results for the energies and lower semicontinuity for the slopes. In the following it is convenient to express all quantities in terms of the linear setting. To this end, recalling \eqref{nonlinear energy} and \eqref{eq: D,D0}, for $u,v \in W^{2,p}_0(\Omega)$ and $\tau,\delta>0$ we define 
\begin{align*}
&  \bar{\phi}_{\delta}(u) = \phi_\delta(\id + \delta u), \ \ \bar{\phi}_{\delta,P}(u) = \delta^{-p\alpha}\int_\Omega P(\delta\nabla^2 u), \ \    \bar{\phi}_{\delta,W}(u) = \bar{\phi}_\delta(u)  - \bar{\phi}_{\delta,P}(u), \\
&\bar{\mathcal{D}}_\delta(u,v) = {\mathcal{D}}_\delta(\id+  \delta u, \id + \delta v), \ \ \ 
\bar{\Phi}_\delta(\tau,v;u)    = \bar{\phi}_\delta(u) + \frac{1}{2\tau}\bar{\mathcal{D}}_\delta(u,v)^2,\\
&|\partial \bar{\phi}_\delta|_{\bar{\mathcal{D}}_\delta}(u) =  |\partial {\phi}_\delta|_{{\mathcal{D}}_\delta}(\id + \delta u).
\end{align*}
We extend $\bar{\phi}_\delta$ to a functional defined on $H^1_0(\Omega) $ by setting $\bar{\phi}_\delta(u) = + \infty$ for $u \in H^1_0(\Omega)  \setminus W^{2,p}_0(\Omega) $. Likewise, we extend $\bar{\Phi}_\delta$. Moreover, we say $u \in \bar{\mathscr{S}}_\delta^M$ if $\id + \delta u \in \mathscr{S}_\delta^M$. We obtain the following $\Gamma$-convergence results. (For an exhaustive treatment of $\Gamma$-convergence we refer the reader to \cite{DalMaso:93}.)

\begin{theorem}[$\Gamma$-convergence]\label{th: Gamma}
Let   $(\delta_n)_n$ be a null sequence. 

(i) The functionals $\bar{\phi}_{\delta_n}: H^1_0(\Omega) \to [0,\infty]$ $\Gamma$-converge to $\bar{\phi}_0$ in the weak $H^1(\Omega)$-topology.

(ii) For each $\tau>0$, $M>0$, and each sequence $(\bar{v}_n)_n$ with $\bar{v}_n \in \bar{\mathscr{S}}_{\delta_n}^M$ and  $\bar{v}_n \to \bar{v}$ strongly in $H^1(\Omega)$, the functionals $\bar{\Phi}_{\delta_n}(\tau,\bar{v}_n;\cdot):  H^1_0(\Omega) \to [0,\infty]$ $\Gamma$-converge to $\bar{\Phi}_0(\tau,\bar{v}; \cdot)$ in the weak $H^1(\Omega)$-topology.

\end{theorem}

\Proof 
(i) The result is essentially proved in  the paper \cite{DalMasoNegriPercivale:02} and we only give a short sketch highlighting the relevant adaptions.  Since $\bar{\phi}_{\delta_n,P} \ge 0$, for the lower bound it suffices to prove $\liminf_{n \to \infty} \bar{\phi}_{\delta_n,W}(u_n) \ge \bar{\phi}_0(u)$ whenever $u_n \rightharpoonup u$ weakly in $H^1(\Omega)$. This was proved under more general assumptions in \cite[Proposition 4.4]{DalMasoNegriPercivale:02}. In our setting it follows readily by using Lemma \ref{lemma: metric space-properties}(iv) and the lower semicontinuity of $\bar{\phi}_0$ (see Lemma \ref{D-lin}(iii)). 

By a general approximation  argument in the theory of $\Gamma$-convergence it suffices to establish the upper bound for smooth functions $u$, cf. \cite[Proposition 4.1]{DalMasoNegriPercivale:02}. For such a function, setting $u_n = u$, we find $\lim_n \bar{\phi}_{\delta_n,W}(u_n) = \bar{\phi}_0(u)$ (see Lemma \ref{lemma: metric space-properties}(iv) or \cite[Proposition 4.1]{DalMasoNegriPercivale:02}) and moreover it is not hard to see that $\bar{\phi}_{\delta_n,P}(u_n) \to 0$ by the growth of $P$ and the fact that $\alpha<1$. This concludes the proof of (i).

(ii) We first suppose that the sequence $(\bar{v}_n)_n$ is constantly $\bar{v}$. Then $\bar{\Phi}_{\delta_n}(\tau,\bar{v};\cdot)$ $\Gamma$-converges to $\bar{\Phi}_{0}(\tau,\bar{v};\cdot)$  repeating exactly the proof of (i), where,  in addition to Lemma \ref{lemma: metric space-properties}(iv), we also use Lemma \ref{lemma: metric space-properties}(iii). To obtain the general case, it now suffices to prove that for every sequence $(u_n)_n$ uniformly bounded in $H_0^1(\Omega)$ and $u_n \in \bar{\mathscr{S}}_{\delta_n}^M$ for some $M$ large enough we obtain
$$\lim\nolimits_{n\to \infty} |\bar{\mathcal{D}}_{\delta_n}(u_n,\bar{v}_{n})^2 - \bar{\mathcal{D}}_{\delta_n}(u_n,\bar{v})^2| = 0.$$
In view of Lemma \ref{lemma: metric space-properties}(iii), it suffices to show $\lim\nolimits_{n\to \infty} |\bar{\mathcal{D}}_{0}(u_n,\bar{v}_{n})^2 - \bar{\mathcal{D}}_{0}(u_n,\bar{v})^2| = 0$.    To this end, we note that (recall \eqref{eq: D,D0})
\begin{align*}
\bar{\mathcal{D}}_{0}(u_n,\bar{v}_{n})^2 - \bar{\mathcal{D}}_{0}(u_n,\bar{v})^2 &= \int_\Omega \C_D[\nabla \bar{v}_{n}, \nabla \bar{v}_{n}] - \int_\Omega \C_D[\nabla \bar{v}, \nabla \bar{v}] \\& \ \ \  - 2\int_\Omega \C_D[\nabla u_n, \nabla \bar{v}_{n} - \nabla\bar{v}],
\end{align*}
which by the assumption on $(\bar{v}_n)_n$ and $(u_n)_n$ converges to zero. \eop 

%{\cred \tt Should we give some more details?  I think it is ok since  such results are quite standard by now.}

We remark that by a general result in the theory of $\Gamma$-convergence we get that (almost) minimizers associated to the sequence of functionals converge to minimizers of the limiting functional.  We obtain the following strong convergence result for recovery sequences which in various settings has been derived in, e.g., \cite{DalMasoNegriPercivale:02, FriedrichSchmidt:2011, Schmidt:08}.

\begin{lemma}[Strong convergence of recovery sequences]\label{lemma: energy}
Suppose that the assumptions of Theorem \ref{th: Gamma} hold. Let $M>0$, let  $(u_n)_n$ be a sequence  with $u_n \in \bar{\mathscr{S}}_{\delta_n}^M$. Let $u \in H_0^1(\Omega)$   such that $u_n\rightharpoonup u$ weakly in $H^1(\Omega)$ and 
$$(i) \ \ \bar{\phi}_{\delta_n}(u_n) \to \bar{\phi}_0(u) \ \ \ \text{or} \  \ \ (ii) \ \ \bar{\Phi}_{\delta_n}(\tau,\bar{v}_n; u_n) \to\bar{\Phi}_0(\tau,\bar{v}; u).$$ 
Then $u_n \to  u$ strongly in $H^1(\Omega)$. 
\end{lemma}

\Proof If $\bar{\phi}_{\delta_n}(u_n) \to \bar{\phi}_0(u)$,  we find $\bar{\phi}_0(u_n) \to \bar{\phi}_0(u)$ by Lemma \ref{lemma: metric space-properties}(iv) and thus by Lemma \ref{D-lin}(iii)
\begin{align*}
\Vert & e(u_n-u) \Vert^2_{L^2(\Omega)} \le C\int_\Omega \C_W[e(u_n-u),e(u_n-u)] \\
&  = C\Big(\int_\Omega \C_W[e(u_n),e(u_n)] + \int_\Omega \C_W[e(u),e(u)] -2\int_\Omega \C_W[e(u_n),e(u)]\Big) \to 0
\end{align*}
as $n \to \infty$. The assertion of (i) follows from Korn's inequality. The proof of (ii) is similar, where one additionally  takes Lemma \ref{lemma: metric space-properties}(iii) into account. \eop

We close this section with a lower semicontinuity result for the slopes.

\begin{lemma}[Lower semicontinuity of slopes]\label{lemma: lsc-slope}
For each sequence $(u_n)_n\subset \bar{\mathscr{S}}_{\delta_n}^M$ with $u_n \rightharpoonup u$ weakly in $H^1(\Omega)$ we have $\liminf_{n \to \infty}|\partial \bar{\phi}_{\delta_n}|_{\bar{\mathcal{D}}_{\delta_n}}(u_n) \ge |\partial \bar{\phi}_0|_{\bar{\mathcal{D}}_0}(u)$. 
\end{lemma}

\Proof For $\eps>0$ fix $u' \in  C^\infty_c(\Omega;\R^d)$ with $\Vert u' - u \Vert_{H^1(\Omega)} \le \eps$.  Fix $v \in C^\infty_c(\Omega;\R^d)$, $v \neq u',u$. We first note that with  $w_n := u_n - u'+ v$ we have by Lemma \ref{lemma: slopes}(i)
\begin{align*}
|\partial \bar{\phi}_{\delta_n}|_{\bar{\mathcal{D}}_{\delta_n}}(u_n) & =  \sup_{w \neq u_n} \  \frac{(\bar{\phi}_{\delta_n}(u_n) - \bar{\phi}_{\delta_n}(w))^+}{\bar{\mathcal{D}}_{\delta_n}(u_n,w)(1 + C  \Vert \Id + \delta_n \nabla u_n - (\Id + \delta_n \nabla w) \Vert_{L^\infty(\Omega)})^{1/2}} \\
&\ge    \frac{(\bar{\phi}_{\delta_n}(u_n) - \bar{\phi}_{\delta_n}(w_n ))^+}{\bar{\mathcal{D}}_{\delta_n}(u_n,w_n)(1 + C_v\delta_n)^{1/2}},
\end{align*}
where $C_v$ is a constant depending also on $v$ and $u'$. Note that, since $u',v$ are smooth, we indeed get  $w_n = u_n - u'+ v \in \bar{\mathscr{S}}_{\delta_n}^M$ for $n$ large enough for some possibly larger $M>0$. Consequently, by Lemma \ref{lemma: metric space-properties}(iii),(iv)
we get
\begin{align}\label{lsc-slope1}
\liminf_{n \to \infty}|\partial \bar{\phi}_{\delta_n}|_{\bar{\mathcal{D}}_{\delta_n}}(u_n)  \ge \liminf_{n \to \infty}   \frac{(\bar{\phi}_{0}(u_n) - \bar{\phi}_{0}(w_n) + \bar{\phi}_{\delta_n,P}(u_n) - \bar{\phi}_{\delta_n,P}(w_n) )^+}{\bar{\mathcal{D}}_{0}(u_n,w_n)}.
\end{align}
Recalling \eqref{linear energy} (for $f \equiv 0 $) we obtain by a direct computation  
\begin{align} 
\lim_{n \to \infty} &\big(\bar{\phi}_0(u_n) - \bar{\phi}_0(u_n - u'+ v)\big) = \lim_{n \to \infty} \big(-\bar{\phi}_0(v-u') - 2\int_\Omega \C_W[e(u_n),e(v-u')] \big)\notag \\
&=  -\bar{\phi}_0(v-u') - 2\int_\Omega \C_W[e(u),e(v-u')]\notag\\& = \bar{\phi}_0(u) - \bar{\phi}_0(v) - \bar{\phi}_0(u'-u)  +2\int_\Omega \C_W[e(u'-u),e(v)].
\end{align}
Moreover, by convexity of $P$ and the definition  $w_n := u_n - u'+ v$ we find
\begin{align}\label{lsc-slope3}
\bar{\phi}_{\delta_n,P}(u_n) - \bar{\phi}_{\delta_n,P}(u_n - u'+ v) \ge \delta_n^{-p\alpha} \int_\Omega \partial_GP(\delta_n\nabla^2w_n) : \delta_n(\nabla^2 u'- \nabla^2 v),
\end{align}
which vanishes as $n \to \infty$ by \eqref{assumptions-P}(iii), H\"older's inequality, $1+ \alpha(p-1) - \alpha p >0$, and  the fact that $\Vert \delta_n \nabla^2 w_n \Vert^p_{L^p(\Omega)} \le CM\delta_n^{p\alpha}$. (The latter follows from $w_n \in \bar{\mathscr{S}}_{\delta_n}^M$.) Combining \eqref{lsc-slope1}-\eqref{lsc-slope3},  using  $\bar{\mathcal{D}}_{0}(u_n,w_n) = \bar{\mathcal{D}}_{0}(v,u')$, and recalling $u_n  \rightharpoonup u$, we get after some calculations
\begin{align*}
\liminf_{n \to \infty}|\partial \bar{\phi}_{\delta_n}|_{\bar{\mathcal{D}}_{\delta_n}}(u_n)&  \ge  \frac{(\bar{\phi}_0(u) - \bar{\phi}_0(v) - \bar{\phi}_0(u'-u)  +2\int_\Omega \C_W[e(u'-u),e(v)])^+}{\bar{\mathcal{D}}_{0}(v,u')}\\
& \ge  \frac{(\bar{\phi}_0(u) - \bar{\phi}_0(v))^+}{\bar{\mathcal{D}}_{0}(v,u)} -C\eps
\end{align*}
for some $C>0$ depending only on $u$, $u'$ and $v$. Letting first $\eps\to 0$ and taking then the supremum with respect to $v$ we get
\begin{align*}
\liminf_{n \to \infty}|\partial \bar{\phi}_{\delta_n}|_{\bar{\mathcal{D}}_{\delta_n}}(u_n)  \ge \sup_{v \in C_c^\infty(\Omega), v \neq u} 
&  \frac{(\bar{\phi}_0(u) - \bar{\phi}_0(v))^+}{\bar{\mathcal{D}}_{0}(v,u)}.
\end{align*}
In view of Lemma \ref{lemma: slopes}(ii), the claim now follows by approximating each $v \in H^1_0(\Omega)$ by a sequence of smooth functions noting that the right hand side is continuous with respect to $H^1(\Omega)$-convergence. \eop

 \subsection{Passage from nonlinear to linear viscoelasticity}

In this section we now give the proof of Theorem \ref{maintheorem3}. For the whole section we fix  a null sequence $(\delta_k)_k$ and sequence of initial data $(y_0^k)_{k\in \N} \subset W^{2,p}_\id(\Omega)$ such that $\delta_k^{-1}(y^k_0 - \id) \to u_0 \in H_0^1(\Omega)$. Moreover, we fix $M>0$ so large that $y_0^k \in \mathscr{S}_{\delta_k}^M$ for $k \in \N$.

\begin{proof}[Proof of Theorem  \ref{maintheorem3}(i)]
Let $\tau>0$ and let $\tilde{Y}_\tau^{\delta_k}$ as in \eqref{ds} be a discrete solution. For each $k \in \N$ we then have the sequence $(U^n_k)_{n \in \N}$ with  $U^n_k= \delta_k^{-1}(\tilde{Y}_\tau^{\delta_k}(n\tau) - \id) \in \bar{\mathscr{S}}_{\delta_k}^M$ for $n \in \N$. We need to show that there exists a sequence $(U^n_0)_{n \in \N}$ with $U^0_0 = u_0$ such that
 $$(i) \ \ U^n_0 = {\rm argmin}_{v \in H^1_0(\Omega)} \bar{\Phi}_0(\tau,U^{n-1}_0; v), \ \ \ \ (ii) \ \ \text{$U^n_k \to U^n_0$ strongly in $H^1(\Omega)$} $$
for all $n \in \N$. We show this property by induction.
 
  Suppose $(U^i_0)_{i=0}^n$ have been found such that the above properties hold. In particular, we note that  (ii) holds for $n=0$ by assumption. We now pass from step $n$ to $n+1$. 
  
As $U^n_k \to U^n_0$ strongly in $H^1(\Omega)$ and thus   by Theorem \ref{th: Gamma}(ii) $\bar{\Phi}_{\delta_k}(\tau,U^{n}_k; \cdot)$ $\Gamma$-converges to  $\bar{\Phi}_0(\tau,U^{n}_0; \cdot)$, we derive by properties of $\Gamma$-convergence that the (unique) minimizer of $\bar{\Phi}_0(\tau,U^{n}_0; \cdot)$, denoted by  $U^{n+1}_0$, is the limit of minimizers of  $\bar{\Phi}_{\delta_k}(\tau,U^{n}_k; \cdot)$. Consequently,  we obtain $U^{n+1}_k \rightharpoonup U^{n+1}_0$ weakly in $H^1(\Omega)$ and  $\bar{\Phi}_{\delta_k}(\tau,U^{n}_k; U^{n+1}_k) \to \bar{\Phi}_0(\tau,U^{n}_0; U^{n+1}_0) $. Thus,  Lemma \ref{lemma: energy} implies that the sequence even converges strongly in $H^1(\Omega)$.  This concludes the induction step.
\end{proof}

In the following let $u$ be the unique element of $MM(\bar{\Phi}_0;u_0)$. 

\begin{proof}[Proof of Theorem  \ref{maintheorem3}(ii)]
We let $\sigma$ be the weak $H^1(\Omega)$-topology. We consider the sequence of metrics $\mathcal{D}_k = \bar{\mathcal{D}}_{\delta_k}$ on $H^1_0(\Omega)$ and the functionals $\phi_k = \bar{\phi}_{\delta_k}$ as well as the limiting objects $\bar{\mathcal{D}}_0$ and $\bar{\phi}_0$. We note that  \eqref{compatibility} is satisfied due to Lemma \ref{lemma: metric space-properties}(iii) and the fact that $\bar{\mathcal{D}}_0$ is quadratic and convex (see Lemma \ref{D-lin}(iii)). Moreover, also \eqref{eq: implication} is satisfied by the $\Gamma$-liminf inequality in Lemma \ref{th: Gamma}(i) and Lemma \ref{lemma: lsc-slope}.

Finally, also \eqref{basic assumptions2} holds. In fact, by the rigidity estimate in Lemma \ref{lemma:rigidity}(i) and \eqref{nonlinear energy}, \eqref{assumptions-W}(iii) we find for all $k \in \N$ and $u \in \bar{\mathscr{S}}_{\delta_k}^M$  letting $y = \id + \delta_k u$
\begin{align}\label{go to limit1}
\begin{split}
\Vert u \Vert^2_{H^1(\Omega)} &= \delta_k^{-2} \Vert y-\id \Vert^2_{H^1(\Omega)}  \le C\delta_k^{-2} \Vert \dist(\nabla y,SO(d)\Vert^2_{L^2(\Omega)} \\&\le C\delta_k^{-2} \phi_{\delta_k}(y) \le  CM.
\end{split}
\end{align}

Now consider a  sequence $(y_k)_k$ of generalized minimizing movements starting from $y_0^k$ with $\delta_k^{-1}(y_0^k-\id) \to u_0$ in $H^1(\Omega)$. For convenience we also introduce the curves $u_k = \delta_k^{-1}(y_k - \id)$. Fix $M>0$ so large that $y_0^k \in \mathscr{S}_{\delta_k}^M$ for $k \in \N$.  As $\bar{\phi}_{\delta_k}(u_k(t)) \le \phi_{\delta_k}(y_k^0)$ for all $t \ge 0$, we get $\sup_k\sup_t (\phi_{\delta_k}(u_k(t)) + \mathcal{D}_k(u_k(t),u_0)) < \infty$ by \eqref{go to limit1} and Lemma \ref{lemma: metric space-properties}(iii). 

Consequently, also   \eqref{eq: abstract assumptions1}(i) holds and \eqref{eq: abstract assumptions1}(ii) is satisfied by the assumption on the initial data and Lemma \ref{lemma: metric space-properties}(iv). Since the slopes are strong upper gradients by Lemma \ref{lemma: slopes}, we can apply Theorem \ref{th:abstract convergence 1} and the existence of a limiting curve of maximal slope follows. As this curve is uniquely given by $u$ (see Theorem \ref{maintheorem2}(ii)), we indeed obtain $u_k(t) \rightharpoonup u(t)$ weakly in $H^1(\Omega)$ for all $t \in [0,\infty)$ up to a subsequence.  Since the limit is unique, we see that the whole sequence converges to $u$ by Urysohn's subsequence principle.

It remains to observe that the convergence is actually strong. This follows from the fact that $\lim_{k \to \infty}\bar{\phi}_{\delta_k}(u_k(t)) = \bar{\phi}_0 (u(t))$  for all $t \in [0,\infty)$ (see Theorem \ref{th:abstract convergence 1}) and    Lemma \ref{lemma: energy}. \end{proof}

\begin{proof}[Proof of Theorem  \ref{maintheorem3}(iii)]
Proceeding as in the previous proof, we see that all assumptions of Theorem \ref{th:abstract convergence 2} are satisfied. Therefore, we get that for any sequence of discrete solutions there is a subsequence converging pointwise weakly in $H^1(\Omega)$ to a curve of maximal slope for $\bar{\phi}_0$ which can again be identified as $u$. The strong convergence as well as the convergence of the whole sequence follow exactly as in the previous proof.
\end{proof}

\subsection{Fine representation of the slopes and  solutions to the equations}

 In this section we derive fine representations for the slopes which will allow us to relate curves of maximal slope with solutions to the equations \eqref{nonlinear equation} and \eqref{linear equation}.

Recall that $\C_D$ as defined in \eqref{linear equation} is a fourth order symmetric tensor inducing   a quadratic form $(F_1,F_2) \mapsto \C_D[F_1,F_2]$ which is positive definite on $\R^{d \times d}_{\rm sym}$ (cf. Lemma \ref{D-lin}). Moreover, it maps $\R^{d \times d}$  to $\R^{d \times d}_{\rm sym}$, denoted by $F \mapsto \C_D F$ in the following. More precisely, the mapping $F \mapsto \C_D F$  from $\R^{d \times d}_{\rm sym}$ to $\R^{d \times d}_{\rm sym}$ is bijective. By $\sqrt{\C_D}$ we denote its (unique) root and by $\sqrt{\C_D}^{-1}$ the inverse of $\sqrt{\C_D}$, both mappings defined on $\R^{d \times d}_{\rm sym}$. We start with a fine representation of the  slope in the linear setting.

\begin{lemma}[Slope in the linear setting]\label{lemma: lin-slope}
There exists a linear differential operator  $\mathcal{L}_0: H^1_0(\Omega;\R^{d}) \to L^2(\Omega;\R^{d \times d}_{\rm sym})$ satisfying ${\rm div} \mathcal{L}_0(u) =  0$ in $H^{-1}(\Omega;\R^d)$  such that  for all $u \in H^1_0(\Omega)$ we have
$$|\partial \bar{\phi}_0|_{\bar{\mathcal{D}}_0}(u)  = \Vert \sqrt{\C_D}^{-1}\big(\C_W e(u) + \mathcal{L}_0(u) \big)  \Vert_{L^2(\Omega)}.$$
Particularly, we note that $|\partial \bar{\phi}_0|^2_{\bar{\mathcal{D}}_0}$ is convex on $H^1_0(\Omega)$.
\end{lemma}

\Proof
Recalling \eqref{linear energy} (for $f \equiv 0 $), \eqref{eq: D,D0}, Definition \ref{main def2}(ii), and Lemma \ref{D-lin} we have
\begin{align}\label{lin-slope1}
|\partial \bar{\phi}_0|_{\bar{\mathcal{D}}_0}(u) &= \limsup_{v \to u} \frac{(\bar{\phi}_0(u) - \bar{\phi}_0(v))^+}{\bar{\mathcal{D}}_0(u,v)}\\
&= \limsup_{v \to u} \frac{ (\int_\Omega  \C_W[e(u), e(u-v)] -  \frac{1}{2}\C_W[e(v-u), e(v-u)])^+} {(\int_\Omega \C_D[e(u-v),e(u-v)])^{1/2}}\notag\\
&= \limsup_{v \to u} \frac{ \int_\Omega  \C_W[e(u), e(u-v)]} {\Vert \sqrt{\C_D} e(u-v) \Vert_{L^2(\Omega)}} = \sup_{w \neq 0} \frac{ \int_\Omega  \C_W[e(u), e(w)]} {\Vert \sqrt{\C_D} e(w) \Vert_{L^2(\Omega)}},\notag
\end{align}
where in the second step we  used  $\int_\Omega \C_W[e(v-u), e(v-u)] / \Vert \sqrt{\C_D} e(u-v) \Vert_{L^2(\Omega)} \to 0$ as $v \to u$. Let $\bar{w}$ be  \MK  the unique \EEE solution to the minimization problem
 $$\min_{v \in H^1_0(\Omega)}  \int_\Omega \Big(\frac{1}{2}|\sqrt{\C_D} e(v)|^2 - \int_\Omega  \C_W[e(u), e(v)]\Big)\ .  $$
 \MK Clearly,  $\bar{w}$ necessarily satisfies \EEE
 $$\int_\Omega \big(\sqrt{\C_D} e(\bar{w}) : \sqrt{\C_D}e(\varphi) - \C_W [e(u),e(\varphi)]\big) = 0 $$
 for all $\varphi \in H^1_0(\Omega)$. This condition can also be formulated as
\begin{align}\label{lin-slope2}
\mathcal{L}_0(u): e( \varphi) =  0 \ \ \forall \varphi \in H^1_0(\Omega), \ \ \text{where} \ \ \mathcal{L}_0(u): = \C_D e(\bar{w})   - \C_We(u).
\end{align}
 As the solution $\bar{w}$ depends linearly on $u$, we also get that $\mathcal{L}_0$ is a linear operator.  By \eqref{lin-slope1} and the property of $\mathcal{L}_0$ we now find
\begin{align*}
|\partial \bar{\phi}_0|_{\bar{\mathcal{D}}_0}(u) &= \sup_{w \neq 0} \frac{ \int_\Omega  (\C_W e(u) + \mathcal{L}_0(u)) : e(w)} {\Vert \sqrt{\C_D} e(w) \Vert_{L^2(\Omega)}} \\&= \sup_{w \neq 0} \frac{ \int_\Omega  \big(\sqrt{\C_D}^{-1}(\C_We(u) + \mathcal{L}_0(u)) \big) : \sqrt{\C_D}e(w)} {\Vert \sqrt{\C_D} e(w) \Vert_{L^2(\Omega)}}\\
& \le  \Vert \sqrt{\C_D}^{-1}(\C_We(u) + \mathcal{L}_0(u))  \Vert_{L^2(\Omega)},
 \end{align*}
 where in the last step we used the Cauchy-Schwartz inequality. On the other hand, by definition of $\mathcal{L}_0$ in \eqref{lin-slope2}, we get
\begin{align*}
|\partial \bar{\phi}_0|_{\bar{\mathcal{D}}_0}(u)& \ge 
 \frac{ \int_\Omega  \big(\sqrt{\C_D}^{-1}(\C_We(u)  + \mathcal{L}_0(u)) \big) : \sqrt{\C_D}e(\bar{w})} {\Vert \sqrt{\C_D} e(\bar{w}) \Vert_{L^2(\Omega)}} \\
 & = \Vert \sqrt{\C_D} e(\bar{w}) \Vert_{L^2(\Omega)} = \Vert \sqrt{\C_D}^{-1}(\C_We(u) + \mathcal{L}_0(u))  \Vert_{L^2(\Omega)}. 
\end{align*} 
 This concludes the proof. \eop

Recall the definition of the symmetric fourth order tensor $H_Y = \frac{1}{2}\partial^2_{F_1^2} D^2(Y,Y)$  for  $Y \in GL_+(d)$ (see before Lemma \ref{lemma: metric space-properties}).  Let $Y \in \R^{d \times d}$ be in a small neighborhood of $\Id$ such that $Y^{-1}$ exists. Similarly to the discussion before Lemma \ref{lemma: lin-slope}, we get that $H_Y$ induces a bijective mapping from $Y^{-\top}\R^{d \times d}_{\rm sym}$ to $Y\R^{d \times d}_{\rm sym}$ by using frame indifference \eqref{eq: assumptions-D}(v) and the growth assumption \eqref{eq: assumptions-D}(vi).  We then introduce $\sqrt{H_Y}$ as a bijective mapping from $Y^{-\top}\R^{d \times d}_{\rm sym}$ to $Y\R^{d \times d}_{\rm sym}$. In a similar fashion, we introduce the inverse $\sqrt{H_Y}^{-1}$.

 For a given deformation $y: \Omega \to \R^d$ we introduce a mapping $H_{\nabla y}: \Omega \to \R^{d \times d \times d \times d}$   by $H_{\nabla y}(x) = H_{\nabla y(x)}$ for $x \in \Omega$. We note by Lemma \ref{lemma:rigidity}(ii), the fact that $D \in C^3$, and a continuity argument that   
\begin{align}\label{continuity for H}
\Vert \sqrt{H_\Id} - \sqrt{H_{\nabla y}} \Vert_{L^\infty(\Omega)}  \le C\delta^\alpha 
\end{align}
 for all $y \in \mathscr{S}_\delta^M$ for a sufficiently large constant $C>0$. Moreover, recall the definition of the operator $\mathcal{L}_P: \lbrace \nabla^2 u: u \in W^{2,p}_\id(\Omega)\rbrace \to W^{-1,\frac{p}{p-1}}(\Omega;\R^{d \times d})$ in \eqref{LP-def}.   We write  $\beta = \delta^{2-\alpha p}$ in the following for convenience.  Note that $\int_\Omega   \partial_GP(\nabla^2 y) : \nabla^2 \varphi = \mathcal{L}_P(y) : \nabla \varphi $ for all $y \in W^{2,p}_\id(\Omega)$ and $\varphi \in W^{2,p}_0(\Omega)$, where the boundary term vanishes due to $\nabla \varphi  =0$ on $\partial \Omega$.  We now obtain the following result.

\begin{lemma}[Slope in the nonlinear setting]\label{lemma: nonlin-slope}
There exists a differential operator $\mathcal{L}^*_P: \lbrace y \in W^{2,p}_\id(\Omega): {\rm div}\mathcal{L}_P(\nabla^2 y) \in H^{-1}(\Omega; \R^{d}) \rbrace \to L^2(\Omega; \R^{d\times d})$ satisfying ${\rm div}\mathcal{L}^*_P(y) =  {\rm div}\mathcal{L}_P(\nabla^2 y)$ in $H^{-1}(\Omega; \R^{d})$ such that for $\delta>0$ small enough and for all $y \in \mathscr{S}_\delta^M$ we have
\begin{align*}
|\partial \phi_\delta|_{{\mathcal{D}}_\delta}(y)  = \begin{cases} 
 \tfrac{1}{\delta}\Vert \sqrt{H_{\nabla y}}^{-1}\big(\partial_FW(\nabla y) + \beta\mathcal{L}^*_P (y) \big)  \Vert_{L^2(\Omega)} & \text{if} \ {\rm div}\mathcal{L}_P(\nabla^2 y) \in H^{-1}(\Omega),\\
 + \infty & \text{else}.
 \end{cases} 
 \end{align*}
\end{lemma}

\begin{rem}\label{rem-slope}
{\normalfont
We remark that the expression is well defined in the following sense: If $\nabla y(x) = Y(x)$ in the above notation, then we indeed have $\partial_FW(\nabla y(x)) + \beta\mathcal{L}^*_P (y(x)) \in Y(x) \R^{d \times d}_{\rm sym}$ for a.e. $x \in \Omega$. 
}
\end{rem}

\Proof We (i) first prove the lower bond in the case  ${\rm div}\mathcal{L}_P(\nabla^2 y) \in H^{-1}(\Omega)$  and (ii) afterwards if ${\rm div}\mathcal{L}_P(\nabla^2 y) \notin H^{-1}(\Omega)$. Finally, (iii)  we establish the upper bound.

(i) Suppose that ${\rm div}\mathcal{L}_P(\nabla^2 y) \in H^{-1}(\Omega)$. Consider the minimization problem
 $$\min_{w \in H^1_0(\Omega)}  \int_\Omega \Big(\frac{1}{2}|\sqrt{H_{\nabla y}}\nabla w|^2 - (\partial_FW(\nabla y)+ \beta\mathcal{L}_P(\nabla^2 y)\Big):  \nabla w.$$
By \eqref{continuity for H}, the fact that $\sqrt{H_\Id} = \sqrt{\C_D}$, Lemma \ref{D-lin}(iii), and Korn's inequality we have
\begin{align*}
\Vert \sqrt{H_{\nabla y}}\nabla w\Vert_{L^2(\Omega)}^2& \ge \Vert \sqrt{H_\Id}\nabla w\Vert_{L^2(\Omega)}^2 - C\delta^{2\alpha} \Vert \nabla w\Vert_{L^2(\Omega)}^2 \\&\ge C\Vert e(w)\Vert_{L^2(\Omega)}^2 - C\delta^{2\alpha} \Vert \nabla w\Vert_{L^2(\Omega)}^2 \ge C\Vert \nabla w\Vert_{L^2(\Omega)}^2
\end{align*} 
for $\delta$ sufficiently small for all $w \in H^1_0(\Omega)$. Moreover, we have $|\int_\Omega \mathcal{L}_P(\nabla^2 y):  \nabla w| \le \Vert {\rm div}\mathcal{L}_P(\nabla^2 y) \Vert_{H^{-1}(\Omega)} \Vert w \Vert_{H^1(\Omega)}$ for all $w \in H^1_0(\Omega)$.  Thus, the solution  $\bar{w}$ of the problem exists, is unique,  and satisfies 
 $$(\partial_FW(\nabla y) + \beta\mathcal{L}_P(\nabla^2 y) ) :  \nabla  \varphi = \sqrt{H_{\nabla y}} \nabla \bar{w} : \sqrt{H_{\nabla y}} \nabla \varphi = H_{\nabla y} \nabla \bar{w} :   \nabla \varphi $$
 for all $\varphi \in H^1_0(\Omega)$.  Define $\mathcal{L}^*_P(y) := \beta^{-1}(  H_{\nabla y} \nabla \bar{w} - \partial_FW(\nabla y))$ and note that 
\begin{align}\label{nonlin-slope2}
\mathcal{L}^*_P(y) : \nabla \varphi  = \mathcal{L}_P(\nabla^2 y) : \nabla \varphi   \ \ \ \text{ for all $\varphi \in H^1_0(\Omega)$}  
\end{align}
 as well as $\mathcal{L}^*_P(y) \in L^2(\Omega)$. Moreover, since $\beta\mathcal{L}^*_P(y) + \partial_FW(\nabla y)=  H_{\nabla y} \nabla \bar{w}$, recalling the properties of $H_{\nabla y}$ we see that Remark \ref{rem-slope}  applies.   Fix $\eps>0$ and choose $w_\eps \in C_c^\infty(\Omega;\R^d)$ with $\Vert \bar{w} - w_\eps \Vert_{H^1(\Omega)} \le \eps$.  Letting $w_n = y - \frac{1}{n}w_\eps$ we get by  a Taylor expansion
 \begin{align*}
 n\delta^2(  \phi_\delta(w_n) - & \phi_\delta(y))  =  n\int_\Omega \partial_FW(\nabla y) : (\nabla  w_n - \nabla y) + n O(\Vert \nabla w_n- \nabla y \Vert^2_{L^2(\Omega)})\notag\\   
 &   +  n \beta \int_\Omega \partial_GP(\nabla^2 y): (\nabla^2w_n - \nabla^2 y)+ n\beta  O(\Vert \nabla^2 w_n - \nabla^2 y \Vert^2_{L^2(\Omega)}) \notag\\ 
 & \ \ \ \  \ \ = - \int_\Omega \partial_FW(\nabla y) : \nabla w_\eps - \beta \partial_GP(\nabla^2 y): \nabla^2w_\eps  + O(1/n), 
 \end{align*}
 where $O(1/n)$ depends on the choice of $w_\eps$. Similarly, we get by Lemma \ref{lemma: metric space-properties}(i)
  \begin{align*}
 n^2\delta^2\mathcal{D}_\delta(y,w_n)^2 &=  n^2 \int_\Omega    H_{\nabla y}[\nabla (y - w_n), \nabla (y - w_n)]  + n^2 O(\Vert   \nabla w_n- \nabla y\Vert^3_{L^3(\Omega)}) \notag \\& =  \Vert \sqrt{H_{\nabla y}} \nabla w_\eps \Vert^2_{L^2(\Omega)}+ O(1/n).  
  \end{align*}
  For brevity we introduce
  $$\Phi(w) =   \Big(\int_\Omega (\partial_FW(\nabla y)+\beta\mathcal{L}_P(\nabla^2 y)) : \nabla w \Big)\Vert \sqrt{H_{\nabla y}} \nabla w \Vert_{L^2(\Omega)}^{-1}. $$
 Since $\mathcal{D}_\delta(y,w_n) \to 0$, we now obtain 
 \begin{align*}
\delta|\partial \phi_\delta|_{\mathcal{D}_\delta}(y)& \ge \limsup_{n \to \infty}  \frac{\delta(\phi_\delta(y) - \phi_\delta(w_n))^+}{\mathcal{D}_\delta(y,w_n)} \\&\ge    \frac{\int_\Omega  \partial_FW(\nabla y) : \nabla w_\eps + \int_\Omega\beta \partial_GP(\nabla^2 y) : \nabla^2 w_\eps  }{\Vert \sqrt{H}_{\nabla y} \nabla w_\eps \Vert_{L^2(\Omega)} } = \Phi(w_\eps)
\end{align*} 
where in the last step we used the definition of $\mathcal{L}_P$  in \eqref{LP-def}. Recalling the definition of $\mathcal{L}^*_P$ and \eqref{nonlin-slope2} we now derive
 \begin{align*}
\Phi(\bar{w}) - \Phi(w_\eps) + \delta|\partial \phi_\delta|_{\mathcal{D}_\delta}(y)& \ge  \Phi(\bar{w})   =  \frac{\int_\Omega H_{\nabla y} \nabla \bar{w}  : \nabla \bar{w}  }{\Vert \sqrt{H_{\nabla y}} \nabla \bar{w} \Vert_{L^2(\Omega)} }   \\&
=  \frac{\int_\Omega \sqrt{H_{\nabla y}} \nabla \bar{w}  :\sqrt{H_{\nabla y}} \nabla \bar{w}  }{\Vert \sqrt{H_{\nabla y}} \nabla \bar{w} \Vert_{L^2(\Omega)} } = \Vert \sqrt{H_{\nabla y}} \nabla \bar{w} \Vert_{L^2(\Omega)} \\
& =  \Vert \sqrt{H_{\nabla y}}^{-1}\big(\partial_FW(\nabla y) + \beta\mathcal{L}^*_P (y) \big)  \Vert_{L^2(\Omega)}.
\end{align*}
By definition of $w_\eps$ we get $|\Phi(\bar{w}) - \Phi(w_\eps)| \to 0$ as $\eps \to 0$ and   the lower bound in the case ${\rm div}\mathcal{L}_P(\nabla^2  y) \in H^{-1}(\Omega)$ follows.

(ii) Now suppose that ${\rm div}\mathcal{L}_P(\nabla^2 y) \notin H^{-1}(\Omega)$. Let   $(y_n)_n$ be  a sequence of smooth functions  converging to $y$ in $W^{2,p}(\Omega)$. Then    $ \mathcal{L}^*_P(y_n)$ is not bounded in $L^2(\Omega)$. Indeed, otherwise we would get by the definition of $\mathcal{L}_P$, \eqref{assumptions-P}(iii), and \eqref{nonlin-slope2} that 
\begin{align*}
\Big|\int_\Omega \mathcal{L}_P(\nabla^2 y) : \nabla \varphi\Big| & = \Big|\int_\Omega \partial_G P(\nabla^2 y): \nabla^2 \varphi\Big| = \lim_{n \to \infty}\Big|\int_\Omega \partial_G P(\nabla^2 y_n): \nabla^2 \varphi\Big| \\
& = \lim_{n \to \infty} \Big|\int_\Omega \mathcal{L}^*_P(y_n) : \nabla\varphi\Big| \le C \Vert \nabla \varphi\Vert_{L^2(\Omega)}
\end{align*} 
for all $\varphi \in W^{2,p}_0(\Omega)$. This, however, contradicts the assumption ${\rm div}\mathcal{L}_P(\nabla^2 y) \notin H^{-1}(\Omega)$. As energy and dissipation are $W^{2,p}(\Omega)$-continuous (see \eqref{assumptions-W},\eqref{assumptions-P}, Lemma \ref{lemma: metric space-properties}(ii)), we find for some fixed $\eps>0$   and $n$ large enough by Lemma \ref{lemma: slopes}(i)  
$$\eps + |\partial \phi_\delta|_{{\mathcal{D}}_\delta}(y) \ge \sup_{w \neq y_n} \frac{(\phi_\delta(y_n) - \phi_\delta(w))^+}{\mathcal{D}_\delta(y_n,w)(1 + C \Vert \nabla y_n - \nabla w \Vert_{L^\infty(\Omega)})^{1/2}} =|\partial \phi_\delta|_{{\mathcal{D}}_\delta}(y_n).$$
By the representation of the slope at $y_n$ and the fact that  $ \mathcal{L}^*_P(y_n)$ is not bounded in $L^2(\Omega)$, the right hand side tends to infinity for $n \to \infty$, as desired.

(iii) For the upper bound, we first   use Lemma \ref{lemma: metric space-properties}(i),(ii), and  Lemma \ref{th: metric space}(ii) to get
\begin{align*}
1 & = \lim_{w \to v} \frac{\mathcal{D}_\delta(v,w)^2}{\mathcal{D}_\delta(v,w)^2} \ge \limsup_{w \to v} \frac{\Vert \sqrt{H_{\nabla v}} \nabla (w-v) \Vert^2_{L^2(\Omega)} -  C\Vert \nabla v - \nabla w \Vert^3_{L^3(\Omega)}}{\delta^2\mathcal{D}_\delta(v,w)^2} \\
& \ge   \limsup_{w \to v} \frac{\Vert \sqrt{H_{\nabla v}} \nabla (w-v) \Vert^2_{L^2(\Omega)} }{\delta^2\mathcal{D}_\delta(v,w)^2} - C\limsup_{w \to v}   \Vert \nabla v - \nabla w \Vert_{L^\infty(\Omega)}\\
& =   \limsup_{w \to v} \frac{\Vert \sqrt{H_{\nabla v}} \nabla (w-v) \Vert^2_{L^2(\Omega)} }{\delta^2\mathcal{D}_\delta(v,w)^2}.
\end{align*}
This together with Lemma \ref{lemma:C1} and the convexity of $P$ gives
\begin{align*}
\delta|\partial \phi_\delta|_{\mathcal{D}_\delta}(y) & =   \limsup_{w \to y}\frac{\delta^2(\phi_\delta(y) - \phi_\delta(w))^+}{\delta\mathcal{D}_\delta(y,w)}   \\&\le \limsup_{w \to y} \frac{\int_\Omega \partial_FW(\nabla y) : \nabla (y -  w) + \int_\Omega \beta \partial_GP(\nabla^2 y): \nabla^2 (y -w)}{\Vert \sqrt{H_{\nabla y}} (\nabla w - \nabla y) \Vert_{L^2(\Omega)} }.
\end{align*}
Recalling the definition of $\mathcal{L}_P$ and using \eqref{nonlin-slope2} as in the lower bound, we get
\begin{align*}
\delta|\partial \phi_\delta|_{\mathcal{D}_\delta}(y) &   \le \limsup_{w \to y} \frac{\int_\Omega (\partial_FW(\nabla y) + \beta\mathcal{L}^*_P(y) ) : \nabla (y -  w)}{\Vert \sqrt{H_{\nabla y}} (\nabla w - \nabla y) \Vert_{L^2(\Omega)} }.
\end{align*}
Finally, the Cauchy Schwartz inequality gives
\begin{align*}
\delta|\partial \phi_\delta|_{\mathcal{D}_\delta}(y) &   \le \limsup_{w \to y} \frac{\int_\Omega \sqrt{H_{\nabla y}}^{-1}(\partial_FW(\nabla y) + \beta\mathcal{L}^*_P(y) ) : \sqrt{H_{\nabla y}}\nabla (y -  w)}{\Vert \sqrt{H_{\nabla y}} (\nabla w - \nabla y) \Vert_{L^2(\Omega)} } \\
& \le \Vert \sqrt{H_{\nabla y}}^{-1}(\partial_FW(\nabla y) + \beta\mathcal{L}^*_P(y)) \Vert_{L^2(\Omega)}. 
\end{align*}
 \eop
 
 Finally, following \cite[Section 1.4]{AGS} we relate curves of maximal slope with solutions to the equations \eqref{nonlinear equation} and \eqref{linear equation}. Similar to \cite[Corollary 1.4.5]{AGS}, this relies on the fact that the stored energy can be written as a sum of a convex functional and a $C^1$ functional on $H^1(\Omega)$.

\begin{proof}[Proof of Theorem  \ref{maintheorem1}(iii) and Theorem  \ref{maintheorem2}(iii)]
 We only give the proof for the nonlinear equation. The proof for the linear equation is easier and can be seen along similar lines. 
 
First, the fact that  $\phi_\delta(y(t))$ is decreasing in time together with \eqref{nonlinear energy}-\eqref{assumptions-P} gives $y \in L^\infty([0,\infty ); W^{2,p}_{\id}(\Omega))$. Moreover, since  $|y'|_{\mathcal{D}_\delta} \in L^2([0,\infty ))$ by \eqref{slopesolution} and $\mathcal{D}_\delta$ is equivalent to the $H^1(\Omega)$-norm (see Lemma \ref{lemma: metric space-properties}(ii)), we observe that $y$ is an absolutely continuous curve in the Hilbert space $H^1(\Omega)$.  By \cite[Remark 1.1.3]{AGS} this implies that $y$ is differentiable for a.e. $t$ with $\partial_t \nabla y(t) \in L^2(\Omega)$ for a.e. $t$, that
\begin{align}\label{eq: derivative in Hilbert}
\nabla y(t) - \nabla y(s) = \int_s^t \partial_t \nabla y(r) \, dr \ \ \ \text{a.e. in $\Omega$ \ \ \  for all }  0 \le s < t \end{align}
and that $y \in W^{1,2}([0,\infty);H^1(\Omega))$.  More precisely, by Fatou's lemma and Lemma \ref{lemma: metric space-properties}(i) we get for a.e. $t$
\begin{align}\label{PDE1}
\begin{split}
|y'|_{\mathcal{D}_\delta}(t) & = \lim_{s \to t}\frac{\mathcal{D}_\delta(y(s),y(t))}{|s-t|} = \lim_{s \to t} \delta^{-1}\Big(\frac{\delta^2\mathcal{D}_\delta(y(s),y(t))^2}{|s-t|^2} \Big)^{1/2} \\&\ge  \delta^{-1} \Big( \int_\Omega \liminf_{s\to t} \Big(  H_{\nabla y(t)} \Big[\frac{\nabla y(s) - \nabla y(t)}{|s-t|}, \frac{\nabla y(s) - \nabla y(t)}{|s-t|}\Big] \\
&  \ \ \ \ \ \ \ \  \ \ \ \ \ \ \ \ \ \ \ \ \  \ \ \ \ \ \ \ \ \ \ \ \ \ \ \ \  \ \ \ \ \ \ \ \ + |s-t|^{-2} O(|\nabla y(t) - \nabla y(s)|^3) \Big)  \Big)^{1/2} \\
& =\delta^{-1} \Big(\int_\Omega H_{\nabla y(t)}[\partial_t \nabla y(t),\partial_t \nabla y(t) ]\Big)^{1/2} = \delta^{-1}\Vert \sqrt{H_{\nabla y(t)}} \partial_t \nabla y(t)\Vert_{L^2(\Omega)}. 
\end{split}
\end{align}
We now determine the derivative  $\frac{{\rm d}}{{\rm d}t} \phi_\delta(y(t))$ of the absolutely continuous curve $\phi_\delta \circ y$.  Fix  $t$ such that  $\lim_{s \to t}\frac{\mathcal{D}_\delta(y(s),y(t))}{|s-t|}$ exists, which holds for a.e. $t$.  Then by Lemma \ref{lemma:C1} we find 
 $$\lim_{s \to \infty}  \frac{\int_\Omega W(\nabla y(s)) - \int_\Omega W(\nabla y(t)) - \int_\Omega \partial_FW(\nabla y(t)) : (\nabla y(s) - \nabla y(t))}{s-t} = 0.$$
The previous estimate  together with the convexity of $P$ yields
\begin{align*}
\frac{{\rm d}}{{\rm d}t} \phi_\delta(y(t)) & = \lim_{s \to t} \frac{\phi_\delta(y(s)) - \phi_\delta(y(t))}{s-t} \\&\ge \liminf_{s \to t}   \frac{1}{\delta^2(s-t)}\int_\Omega \Big(\partial_FW(\nabla y(t)) : (\nabla y(s) - \nabla y(t)) \\& \ \ \ \ \  \ \ \  \  \ \ \ \ \ \ \ \ \ \ \ \ \ \ \ \ \ \  \ \ \ \  + \beta \partial_GP(\nabla^2 y(t)) :  (\nabla^2 y(s) - \nabla^2 y(t))  \Big)\\ 
& = \liminf_{s \to t}   \frac{1}{\delta^2(s-t)}\int_\Omega \big(\partial_FW(\nabla y(t))+ \beta\mathcal{L}^*_P(y(t)) \big) : (\nabla y(s) - \nabla y(t)),  
\end{align*}
where as before $\beta = \delta^{2-\alpha p}$. In the last step we integrated by parts and used ${\rm div}(\mathcal{L}^*_P(y(t))) =  {\rm div}(\mathcal{L}_P(\nabla^2 y(t)))$ by Lemma \ref{lemma: nonlin-slope}. Note that the last term is well defined as  $\mathcal{L}^*_P(y(t)) \in L^2(\Omega)$ for a.e. $t$ by Lemma \ref{lemma: nonlin-slope}  and \eqref{slopesolution}. Now \eqref{eq: derivative in Hilbert} implies 
\begin{align*}
\frac{{\rm d}}{{\rm d}t} \phi_\delta(y(t)) & \ge \delta^{-2}\int_\Omega \sqrt{H_{\nabla y(t)}}^{-1}\big(\partial_FW(\nabla y(t)) + \beta\mathcal{L}^*_P(y(t)) \big): \sqrt{H_{\nabla y(t)}}\partial_t\nabla y(t). 
\end{align*}
We find by Lemma \ref{lemma: nonlin-slope}, \eqref{PDE1}, and Young's inequality 
$$  \frac{{\rm d}}{{\rm d}t} \phi_\delta(y(t)) \ge - \frac{1}{2}  \big(|\partial \phi_\delta|^2_{\mathcal{D}_\delta}(y(t)) +  |y'|^2_{\mathcal{D}_\delta}(t)\big) \ge \frac{{\rm d}}{{\rm d}t} \phi_\delta(y(t)),$$
where  the last step is a consequence of the fact that $y$ is a curve of maximal slope with respect to $\phi_\delta$. Consequently, all inequalities employed in the proof are in fact equalities and we get
$$  \sqrt{H_{\nabla y(t)}}^{-1}\big(\partial_FW(\nabla y(t)) + \beta\mathcal{L}^*_P(y(t)) \big) =  -\sqrt{H_{\nabla y(t)}}\partial_t\nabla y(t) 
 $$
 pointwise a.e. in $\Omega$. Equivalently, recalling $\partial_{\dot{F}}R(F,\dot{F}) = \frac{1}{2} \partial^2_{F_1^2} D^2(F,F)\dot{F} = H_{F}\dot{F}$ from  \eqref{intro:R}, we obtain 
$$   \big(\partial_FW(\nabla y(t)) + \beta\mathcal{L}^*_P(y(t)) \big)  + \partial_{\dot{F}}R(\nabla y(t),\partial_t \nabla y(t))  =0$$
 pointwise a.e. in $\Omega$.   Multiplying the equation with $\nabla \varphi$ for $\varphi \in W_0^{2,p}(\Omega)$, using again $\int_\Omega \mathcal{L}^*_P(y(t)) : \nabla \varphi =  \int_\Omega\mathcal{L}_P(\nabla^2 y(t)) : \nabla \varphi$ by Lemma \ref{lemma: nonlin-slope}, and the definition of $\mathcal{L}_P(\nabla^2 y(t)) $,  we conclude that $y$ is a weak solution (see \eqref{nonlinear equation2}). 
 \end{proof}

\noindent \textbf{Acknowledgements} This work has been funded by the Vienna Science and Technology Fund (WWTF)
through Project MA14-009. M.F.~acknowledges support by the Alexander von Humboldt Stiftung and thanks for the warm hospitality at \'{U}TIA AV\v{C}R, where this project has been initiated. M.K.~acknowledges support by the GA\v{C}R-FWF project 16-34894L.  Both authors were also supported   by the M\v{S}MT \v{C}R  mobility project 7AMB16AT015.  We wish to thank Ulisse Stefanelli for turning \BBB our \EEE attention to this problem.

%--------------------------------------------------------------------------

%--------------------------------------------------------------------------
%--------------------------------------------------------------------------
 \typeout{References}

\end{document}